\documentclass[11pt]{amsart}
\usepackage{amsmath,amssymb,amsthm}
\RequirePackage{color}
\definecolor{my-link}{rgb}{0.5,0.0,0.0}
\definecolor{my-blue}{rgb}{0.0,0.0,0.6}
\definecolor{my-red}{rgb}{0.5,0.0,0.0}
\definecolor{my-green}{rgb}{0.0,0.5,0.0}
\RequirePackage[colorlinks, urlcolor=my-blue,linkcolor=my-link,citecolor=my-green]{hyperref}
\usepackage[numbers,sort&compress,square,comma]{natbib}

\usepackage{scalerel}    
\usepackage{stmaryrd}   
\usepackage{mathabx}   

\newtheorem{theorem}{\sc Theorem}[section]
\newtheorem{lemma}[theorem]{\sc Lemma}
\newtheorem{proposition}[theorem]{\sc Proposition}
\newtheorem{corollary}[theorem]{\sc Corollary}

\numberwithin{equation}{section}

\newcommand{\be}{\begin{equation}}
\newcommand{\ee}{\end{equation}}
\newcommand{\nn}{\nonumber}
\providecommand{\abs}[1]{\vert#1\vert}
\providecommand{\norm}[1]{\Vert#1\Vert}

\newcommand{\fl}[1]{\lfloor{#1}\rfloor}

\def\E{\bE}
\def\P{\bP} 
\def\wlev{\omvec} 
\def\shift{\theta}  
\def\jump{v} 
\def\flux{H}   

\def\cF{\mathcal{F}}

\def\cN{\mathcal{N}}

\def\cS{\mathcal{S}}
\def\cI{\mathcal{I}}

\def\kS{\mathfrak{S}}

\def\bE{\mathbb{E}}  \def\E{\mathbb{E}}
\def\bN{\mathbb{N}}
\def\bP{\mathbb{P}}

\def\N{\mathbb{N}}
\def\P{\mathbb{P}}

\def\bR{\mathbb{R}}   \def\R{\mathbb{R}}
\def\bZ{\mathbb{Z}} \def\Z{\mathbb{Z}}
\def\Zb{\mathbb{Z}}

  \def\Zvec{\mathbf{Z}} \def\thvec{\boldsymbol{\theta}}

\def\w{\omega}
\def\om{\omega}

\def\omvec{\bar\omega}
\def\e{\varepsilon}
\def\ind{\mathbf{1}}
\def\ddd{\displaystyle} 

\def\mE{\mathbf{E}}   

\def\mP{\mathbf{P}}
\def\m1{\mathbf{1}}
\font \mymathbb = bbold10 at 11pt
\newcommand{\one}{\mbox{\mymathbb{1}}}

\def \sqn{\sqrt{n}}

  \def\wt{\widetilde} 

\newcommand\abullet{{\scaleobj{0.6}{\bullet}}}  

\DeclareMathOperator{\Var}{Var}   \DeclareMathOperator{\Cov}{Cov}  
 \def\Vvv{{\rm\mathbb{V}ar}}  \def\Cvv{{\rm\mathbb{C}ov}} 
 \def\hfun{h}
 \def\cJ{\mathcal{J}}  \def\wt{\widetilde}
 \def\qvar{\sigma_0}
  \def\chfn{\lambda}  
 \def\bT{\mathbb{T}} 
 
   \def\qhom{\bar q}   
 \def\chfnhom{\bar\lambda}   
  \def\Yhom{\bar Y}   
  \def\Phom{\bar P}    \def\Ehom{\bar E}   
 \def\Ghom{\bar G}    
 \def\tauhom{\bar\tau}   
  \def\ahom{\bar a}   
  \def\Fhom{\bar F}  \def\ghom{\bar g}

\def\pss{\Gamma}  

\allowdisplaybreaks[1]

\begin{document}

\title[Particles in a dynamic environment]
{
Independent particles 
in a dynamical \\ random environment 
}
\author[M.~Joseph]{Mathew Joseph}
\address{Mathew Joseph\\ Indian Statistical Institute\\ 
Bangalore, 8th Mile Mysore Road, RVCE  Post\\ 
Bengaluru, Karnataka 560059\\ India.}
\email{m.joseph@isibang.ac.in}
\author[F.~Rassoul-Agha]{Firas Rassoul-Agha}
\address{Firas Rassoul-Agha\\ University of Utah\\ 
Department of Mathematics\\ 155 South 1400 East\\  
Salt Lake City, UT 84112\\ USA.}
\email{firas@math.utah.edu}
\urladdr{http://www.math.utah.edu/~firas}
\thanks{F.\ Rassoul-Agha was partially supported by NSF grant DMS-1407574 and Simons Foundation grant 306576.}
\thanks{M.\ Joseph and F.\ Rassoul-Agha were partially supported by 
NSF grant DMS-0747758.} 
\author[T.~Sepp\"al\"ainen]{Timo Sepp\"al\"ainen}
\address{Timo Sepp\"al\"ainen\\ University of Wisconsin-Madison\\ 
Mathematics Department\\ Van Vleck Hall\\ 480 Lincoln Dr.\\  
Madison WI 53706-1388\\ USA.}
\email{seppalai@math.wisc.edu}
\urladdr{http://www.math.wisc.edu/~seppalai}
\thanks{T.\ Sepp\"al\"ainen was partially supported by 
NSF grants DMS-0701091,  DMS-1003651, DMS-1306777 and DMS-1602486,   by Simons Foundation grant 338287, 
 and by the Wisconsin Alumni Research
Foundation.} 
\keywords{random walk in random environment, particle current, 
limit distribution, fractional Brownian motion, EW universality}
\subjclass[2000]{60K35, 60K37} 
\let\thefootnote\relax\footnotetext{{\it Received:} December 1, 2017. {\it Revised:} September 19, 2018.}
\begin{abstract} We study the motion of independent particles in a 
dynamical random environment on the 
integer lattice.  The environment has a product distribution.
 For the multidimensional case, we characterize the class of spatially ergodic invariant measures.
These invariant distributions are mixtures of inhomogeneous Poisson
product measures that depend on the past of the environment. We also 
investigate the correlations in this measure. For dimensions one and two, we prove convergence to equilibrium from spatially ergodic initial distributions.  
In the one-dimensional situation we study fluctuations of the net current seen by an observer
traveling at a deterministic speed.   When this current is centered by its quenched mean
its   limit  distributions are the same 
as for classical independent particles. 
\end{abstract}
\maketitle

\section{Introduction and results}
This paper studies particles that move 
on the integer lattice $\bZ^d$.  Particles interact through
a common environment that specifies their transition 
probabilities in space and time. 
The environment is picked
 randomly at the outset and fixed for all time.
  Given the environment,
 particles evolve independently, governed by 
the transition probabilities specified by the 
environment. 

We have two types of results.  First we characterize those
invariant distributions for the particle process that
satisfy a spatial translation invariance. These turn out
to be mixtures of inhomogeneous Poisson  product measures that 
depend on the past of the environment.
Poisson is expected,
 in view of the classical result 
that a system of
independent random walks has  a homogeneous
Poisson invariant distribution \cite[Section VIII.5]{doob}.
For $d=1,\, 2$, we use coupling ideas from \cite{ekha-gray}
(as presented  in \cite{sepp-exclbook})
to  prove convergence 
to this equilibrium from spatially invariant initial
distributions. 

In the one-dimensional case 
we study   fluctuations of the particle
current seen by an observer moving at the characteristic
speed.   
 In the present setting the characteristic speed
is simply the mean speed $v$ of the particles.  More generally,
the characteristic
speed  is the derivative $\flux'(\rho)$ of the flux $\flux$ as a function
of particle density $\rho$.  The flux $\flux(\rho)$ is the
mean rate of flow across a fixed bond of the lattice
when the system is stationary with density $\rho$.
For independent particles $\flux(\rho)=v\rho$. 

It is expected, and supported
by known rigorous results,   that the current
fluctuations   are of order $n^{{1}/{4}}$ with Gaussian
limits if the macroscopic flux $\flux$ is linear, 
and of order $n^{{1}/{3}}$ with Tracy-Widom type
 limits  if the flux $\flux$  is 
strictly convex or concave.  In statistical physics terminology, the
former is the Edwards-Wilkinson (EW) universality class, and the latter the  
Kardar-Parisi-Zhang (KPZ) universality class. (See \cite{bara-stan} for the
physics perspective on these matters, and \cite{corw-12-rev, sepp-10-ens} for
mathematical reviews.)  
Our motivation is  to investigate the effect of a random environment
 in the EW class.
We find   that,   when the current is centered by its quenched mean and
the environment 
is averaged out,  the fluctuation picture in the 
dynamical environment is the same as that for 
classical independent random walks  \cite{kuma-08, sepp-rw}.  
Consistent with  EW universality,  the current fluctuations  have magnitude
$t^{1/4}$ and occur on a spatial scale of  $t^{1/2}$ where $t$ denotes the 
macroscopic time variable.

There is an interesting contrast with the case of static environment
investigated in  \cite{pete-sepp-10}.  In the static environment, the quenched mean
of the current has fluctuations of magnitude  $t^{1/2}$ and  converges weakly to a Brownian motion.  
Our results suggest that under a dynamic environment  the quenched mean
of the current has fluctuations of magnitude  $t^{1/4}$ and that when the particle system is stationary in time these fluctuations 
are governed by a fractional Brownian motion with Hurst parameter $1/4$.  

Other work on the motion of  independent particles in a random environment includes articles \cite{pete-10,pete-jara-17}.

We turn to a  description of the process and then the results.

\subsection{The particle process and its invariant distributions} 

The particles follow independent random walks in a common dynamical random environment (RWRE). More precisely, they move in a space-time environment 
$\om=(\om_{x,s})_{(x,s)\in \bZ^d\times\bZ}$
indexed by a discrete time variable $s$ 
and a discrete space variable $x$. 
The environment at space-time point $(x,s)\in \bZ^d\times\bZ$ 
is a vector 
$\om_{x,s}=(\om_{x,s}(z):z\in\bZ^d, \, \abs{z}\le R)$ of jump probabilities that satisfy
\be
0\le \om_{x,s}(z)\le 1 \quad\text{and}\quad   
\sum_{z\in\bZ^d\,:\,\abs{z}\le R} \om_{x,s}(z)=1. \label{step-bd}\ee
$R$ is a fixed finite constant that specifies the range of jumps. 
From a space-time point $(x,s)$ admissible jumps are to 
points $(y,s+1)$ such that $ \abs{y-x}\le R$. 
In environment $\om$  the 
   transition probabilities governing the motion of a $\bZ^d$-valued walk
$X_\abullet=(X_s)_{s\in\bZ_+}$ are 
 \be
P^\om[ X_{s+1}=y \,\vert\, X_s=x]= \pi^\om_{s,s+1}(x,y)
\equiv \w_{x,s}(y-x). \label{trans}\ee 
$P^\om$ is the quenched probability measure on the path space
of the walk $X_\abullet$. 
The environment is ``dynamical'' because
at each time $s$  the particle sees a new environment 
$\wlev_s=(\om_{x,s}:x\in\bZ^d)$. 

 $(\Omega, \kS)$ denotes the space of environments $\om$ satisfying the above
 assumptions,   
endowed with the product topology and its Borel $\sigma$-algebra $\kS$.  
The environment restricted to levels  $s\in\{m,\dotsc,n\}$
is denoted by 
\[ \wlev_{m,n}=(\wlev_s)_{m\le s\le n}
=(\om_{x,s}:m\le s\le n, x\in\bZ^d).\]  
Environments at levels generate $\sigma$-algebras
$\kS_{m,n}=\sigma\{\wlev_{m,n}\}$.
In these formulations $m=-\infty$ or $n=\infty$ are also
possible.  
 $T_{x,s}$ is the  shift on $\Omega$, that is 
$(T_{x,s}\om)_{y,t}=\om_{x+y,s+t}$. 

Let $\bP$ be a probability measure on $\Omega$ such that 
\be
\text{the probability vectors $(\om_{x,s})_{(x,s)\in \bZ^d\times\bZ}$ 
are i.i.d.~under $\bP$.}
\label{Pass}\ee

We make two nondegeneracy assumptions. The first one guarantees
that the quenched walk is not degenerate: 
\be
\bP\{ \,\exists\,z\in\bZ^d: \,0<\om_{0,0}(z)<1\,\}>0. 
\label{ass:q00}\ee
Denote the mean transition
 kernel by $p(u)=\E\pi_{s,s+1}(x,x+u)$. 
The second key assumption is that 
\be\begin{array}{l}
\text{there does not exist $x\in\bZ^d$ and an additive   subgroup}\\
\text{$\mathbb G\subsetneq\bZ^d$ such that $\sum_{z\in \mathbb G}p(x+z)=1$.}\\
\end{array}
\label{Rass}\ee
Another way to state assumption \eqref{Rass} is that the averaged walk 
has span 1, or that it is aperiodic in Spitzer's \cite{spitzer} terminology.
 
 To create a system of particles, let 
$\{X^{u,j}_\abullet: u\in\bZ^d, \,j\in\bN\}$ denote  a collection 
of random walks on $\bZ^d$ 
 such that walk $X^{u,j}_\abullet$ starts at site
$u$:  $X^{u,j}_0=u$. When the environment
$\om$ is fixed,  we use $P^\om$ to denote the joint quenched measure of
the walks $\{X^{u,j}_\abullet\}$. 
 Under   $P^\om$ these  walks move independently
on $\bZ^d$ and  each walk  
 obeys transitions  \eqref{trans}.

Further,   assume given an initial
configuration $\eta=(\eta(u))_{u\in\bZ^d}$ of occupation
variables.  Variable $\eta(u)\in\bZ_+$ specifies the number
of particles initially at site $u$.  
$P^\om_\eta$ denotes the quenched distribution of the walks
$\{X^{u,j}_\abullet: u\in\bZ^d, \,1\le j\le\eta(u)\}$.
Occupation variables for all times $s\in\bZ_+$ are then defined
by
\[
\eta_s(x)= \sum_{u\in\bZ^d} \sum_{j=1}^{\eta(u)} \ind\{X^{u,j}_s=x\},
\qquad  (x,s)\in\bZ^d\times\bZ_+. \]
When the initial configuration $\eta=\eta_0$ has
probability distribution $\nu$ we write 
$P^\om_\nu(\cdot)=\int P^\om_\eta(\cdot)\,\nu(d\eta)$ for the quenched  distribution of the process.  

When the environment is averaged over we drop the superscript 
$\om$: for any event $A$ that involves the walks and occupation
variables, and any event $B\subseteq\Omega$, 
$P_\nu(A\times B)=\int_B P^\om_\nu(A)\, \bP(d\om).$   

It will be convenient to construct 
  initial distributions $\nu=\nu^\om$ as functions of  the environment, so that the quenched process distribution is then $P^\om_{\nu^\om}(\cdot)=\int P^\om_\eta(\cdot)\,\nu^\om(d\eta)$.  But then it will always be the
case that $\nu^\om$ depends only on the {\sl past}  $\wlev_{-\infty,-1}$ 
of the environment.  Consequently the initial distribution $\nu^\om$ and
the quenched distribution of the walks $P^\om( \{X^{x,j}_\abullet\}\in\cdot\,)$ are
independent under the product measure  $\P$ on the environment.   The averaged process distribution is then  
\[ \int_\Omega P^\om_{\nu^\om}(\cdot )\, \bP(d\om)
= \int_{\Z_+^{\Z^d}} P_\eta(\cdot)\,\bar\nu(d\eta) = P_{\bar\nu}(\cdot)  \]
where  $\bar\nu(d\eta)=\int_\Omega \nu^\om(d\eta)\, \bP(d\om)$ is the averaged initial distribution.    In particular, in both quenched and  averaged sense,   the initial occupation variables $\{\eta_0(x)\}$ are
independent of the walks $\{X^{x,j}_\abullet\}$.

\medskip

The first result describes the invariant distributions
of the occupation process $\eta_t=(\eta_t(x))_{x\in\bZ^d}$. 
The starting point is an invariant distribution for 
the environment process seen by a tagged particle:
this is the  process $T_{X_n,n}\om$ where  
 $X_\abullet$ denotes a walk that starts at the origin.
A familiar martingale argument and Green function bounds (Proposition
\ref{finvprop} in Section \ref{inv} below) 
 show the existence of an $\kS_{-\infty,-1}$-measurable 
 density function $f$ on $\Omega$ such that  $\E(f)=1$, $\E(f^2)<\infty$, and 
the probability  measure $\bP_\infty(d\om)=f(\om)\,\bP(d\om)$
is invariant for the Markov chain $T_{X_n,n}\om$.

For $0\le \lambda<\infty$ let $\pss^\lambda$ 
denote the mean $\lambda$ 
Poisson distribution on $\bZ_+$. For $0\le \rho<\infty$
and  $\om\in\Omega$
define the following inhomogeneous Poisson product 
 probability distribution on particle
configurations $\eta=(\eta(x))_{x\in\bZ^d}$:
\be
\mu^{\rho,\om}(d\eta)=\bigotimes_{x\in\bZ^d} \pss^{\rho f(T_{x,0}\om)}\big(d\eta(x)\big). 
\label{def:muom}\ee
(Such a measure is called a Cox process with random intensity $\rho f(T_{x,0}\om)$.)
Define the averaged measure by 
\be
\mu^{\rho}=\int \mu^{\rho,\om}\,\P(d\om).
\label{defmuu}\ee

\begin{theorem} Let the dimension $d\ge 1$.
Consider independent particles on $\bZ^d$  in an i.i.d.~space-time environment 
{\rm (}as indicated in  {\rm\eqref{Pass}}{\rm)},
with bounded jumps, under assumptions {\rm\eqref{ass:q00}} and {\rm\eqref{Rass}}.

{\rm (a)}   For each $0\le \rho<\infty$, $\mu^{\rho}$ is the
unique invariant distribution for the process $\eta_\abullet$ 
that is also invariant and ergodic under spatial translations
and has mean occupation $\int\eta(x)\,d\mu^{\rho}=\rho$. 
Furthermore,  the tail $\sigma$-field of the state space $\bZ_+^{\bZ^d}$ is trivial under $\mu^\rho$,
and under the path measure $P_{\mu^\rho}$  the process $\eta_\abullet$ is ergodic
under  time shifts. 

{\rm (b)} Suppose $d=1$ or $d=2$. Let 
 $\nu$ be a probability distribution on 
$\bZ_+^{\bZ^d}$ that is stationary and ergodic under spatial translations and has mean occupation $\rho=\int \eta(x)\,d\nu$.  Then if $\nu$ is the
 initial distribution for the process $\eta_\abullet$,     the process converges in
distribution to the invariant distribution with density $\rho$:  
 $P_\nu\{\eta_t\in\cdot\} \Rightarrow \mu^{\rho}$ as $t\to\infty$. 
\label{invthm}\end{theorem}

Part (b) of the theorem is restricted to   $d=1,2$ because our proof uses recurrence of random walks (see Proposition \ref{rho1gerho2} below).


Two auxiliary  Markov transitions $q$ and $\qhom$ on $\bZ^d$  
 play important roles throughout much of  the paper: 
\be\begin{aligned}
q(x,y)&= \sum_{z\in\Zb^d} \E[\om_{0,0}(z)\om_{x,0}(z+y)]  \\[4pt]
&=\begin{cases} \sum_{z\in\Zb^d} \E[\om_{0,0}(z)\om_{0,0}(z+y)] 
&x=0, y\in\Zb^d\\[4pt]
 \sum_{z\in\Zb^d} p(z)p(z+y-x)
&x\neq 0,y\in\Zb^d  \end{cases} 
\end{aligned}\label{qxy}\ee
and 
\be    \qhom(x,y)=\qhom(0,y-x)=
\sum_{z\in\Zb^d} p(z)p(z+y-x). 
  \label{qbarxy}\ee
Think of $q$ as a symmetric random walk  whose transition probability is  
perturbed at the origin, and of $\qhom$ as the corresponding unperturbed homogeneous
walk. 

For $\theta\in\bT^d=(-\pi,\pi]^d$ define characteristic functions 
\be  \phi^\om(\theta) =\sum_z \om_{0,0}(z)e^{i\theta\cdot z}, \label{chfn1a}\ee
\be
\chfn(\theta)=\sum_{z\in\bZ^d} q(0,z)e^{i\theta\cdot z}= \bE\abs{\phi^\om(\theta)}^2 
\label{chfn2a}\ee
and 
\be \chfnhom(\theta)=\sum_{z\in\bZ^d} \qhom(0,z)e^{i\theta\cdot z}=  \abs{\bE\phi^\om(\theta)}^2.  
\label{barchfn2a}\ee
(We use the bar notation   for   quantities associated with the homogeneous walk $\qhom$, in addition to a few other particular items  such as $\wlev_s$ for the environment on level $s$.  In the case of $\chfnhom(\theta)$ this must not be confused with complex conjugation.) 
 Assumption \eqref{Rass} implies that  the random walk 
$\qhom$ is not supported on a subgroup
smaller than $\bZ^d$, hence  ${\chfnhom(\theta)}<1$ for $\theta\in\bT^d\setminus\{0\}$
\cite[p.~67, T7.1]{spitzer}.
Define a constant $\beta$   by 
\be \beta= \frac1{(2\pi)^d}\int_{\bT^d} 
\frac{1-\chfn(\theta)}{1-\chfnhom(\theta)}\,d\theta.\label{defbeta}\ee
The distribution  $q(0,z)$ is not degenerate by assumption \eqref{ass:q00} and hence
$\chfn(\theta)$ is not identically $1$.  Since also $\chfnhom(\theta)\le\chfn(\theta)$,  we see that
  $\beta\in(0,1]$ is well-defined.

Under the invariant distribution $\mu^\rho$ the covariance of the occupation variables is 
\be\label{cov14} 
\Cov^{\mu^\rho}[\eta(0),\eta(m)] =\rho^2\Cvv[f(\om), f(T_{m,0}\om)]
=\rho^2\E[f(\om) f(T_{m,0}\om)]-\rho^2, \quad m\in\bZ^d. 
\ee
The first equality above comes from the structure of $\mu^\rho$: given $\om$, 
the occupation variables are independent with means 
$E^{\mu^{\rho, \om}}[\eta(m)]=\rho f(T_{m,0}\om)$.  
 Our next theorem gives  a formula for \eqref{cov14}.  

\begin{theorem}  Let $d\ge 1$. 
For $m\in\bZ^d\setminus\{0\}$ 
\begin{align}
  \Cvv[f(\om), f(T_{m,0}\om)]  
= -\,\frac{\beta^{-1}}{(2\pi)^d}\int_{\bT^d} \cos(\theta\cdot m)
\frac{1-\chfn(\theta)}{1-\chfnhom(\theta)}\,d\theta\label{covfm2} \end{align}
 and 
 \be  \Vvv[f(\om)]=\beta^{-1}-1. \label{varf2}\ee
\label{covthm}\end{theorem}

The compact analytic formulas \eqref{defbeta} and \eqref{covfm2}  arise from 
probabilistic formulas that involve the transitions $q$ and $\qhom$ and the potential
kernel of $\qhom$.  The probabilistic  arguments  
are    somewhat different in the recurrent ($d\le 2$) and transient  ($d\ge 3$)
cases.    The reader can find these in Section \ref{covarsec}. 

By the Riemann-Lebesgue lemma we have that
	\[\lim_{m\to\infty}\Cvv[f(\om), f(T_{m,0}\om)] =0.\]
By computing the integral in \eqref{covfm2} an interesting special case arises: 

\begin{corollary}  \label{cor:covf}
For the simplest case where $d=1$ 
and  $p(x)+p(x+1)=1$ for some $x\in\bZ$,  the fixed time occupation variables 
in the stationary process 
are uncorrelated:   \[\Cvv[f(\om), f(T_{m,0}\om)] =0\quad\text{for }m\ne 0.\] 
\end{corollary}


\subsection{Limit of the current process} 

To study the 
 particle current  we restrict to dimension $d=1$.
Define the mean and variance of  the averaged walk by 
 \be v=\sum_{x\in\bZ}  xp(x) 
\quad\text{and}\quad  \sigma^2=\sum_{x\in\bZ} x^2p(x)- v^2.\label{defsigma}\ee
 For $t\in\bR_+=[0,\infty)$ and $r\in\bR$, let 
   \be Y_n(t,r)= \sum_{x>0}  \sum_{j=1}^{\eta_0(x)} 
\mathbf{1}\lbrace X^{x,j}_{\fl{nt}}
\le \fl{nvt}+\fl{r\sqn}\,\rbrace
- \sum_{x\le 0} \sum_{j=1}^{\eta_0(x)} 
\mathbf{1}\lbrace X^{x,j}_{\fl{nt}}
> \fl{nvt}+\fl{r\sqn}\,\rbrace.
 \label{Y}\ee  
  $Y_n(t,r)$ represents  the net 
right-to-left current of particles seen  by a moving observer who starts at the origin and travels to   
$\fl{nvt}+\fl{r\sqn}$  in
 time $\fl{nt}$.

We look at the current under the following assumptions. 
Given $\om$, initial occupation variables obey a product measure
that may depend on the past of the environment, but
so that shifts are respected.   Precisely, 
 \be\begin{array}{l}
\text{given the environment $\om$, initial occupation variables 
 $(\eta_0(x))_{x \in \bZ}$ have}\\
\text{distribution  $\mu^\om(d\eta_0)=\underset{x \in \bZ}\otimes \mu_x^{\om}(d\eta_0(x)) $ 
where $\mu_x^{\om}$ is allowed to  depend  }\\
\text{measurably on   $\wlev_{-\infty,-1}$. Furthermore, 
$\mu^\om_x=\mu^{T_{x,0}\om}_0$. }
\end{array}\label{Dass}\ee
Let $P^\om$ denote the quenched distribution $P^\om_{\mu^\om}$ 
of initial  occupation variables and walks, and $P=P^\om(\cdot) \P(d\om)$
the distribution over everything: particles, walks and environments.  

Make this moment assumption: 
\be E[\eta_0(0)^2] <\infty. 
\label{momass}\ee
Parameters that appear in the results are 
\be\rho_0= E[\eta_0(x)]   \quad \text{and}\quad 
\qvar^2=\E[\Var^\om(\eta_0(0))].    \label{Mass}\ee 

 Next we describe the limiting process. 
Let $\dot W$ be space-time white noise corresponding and $B$  a 
two-sided one-parameter Brownian motion on $\bR$, independent of $\dot W$.
Let $W$ be the two-parameter  Brownian motion on $\bR_+ \times \bR$ given by $W(t,r)=\dot W([0,t]\times[0,r])$, if $r>0$, and $W(t,r)=-\dot W([0,t]\times[r,0])$, if $r<0$.
Define the process $Z(t,r)$ as the unique mild solution of
the stochastic heat equation (see \cite{walsh})
	\[Z_t=\frac{\sigma^2}2Z_{rr}+\sqrt{\rho_0}\,\dot W,\qquad Z(0,r)=\sigma_0 B(r).\]
Process $Z$ is given by 
\be\begin{aligned}
 Z(t,r)&= \sqrt{\rho_0} \iint_{[0,t]\times \bR} 
\varphi_{\sigma^2(t-s)}(r-x)\,dW(s,x) \\
&\qquad  +\;\qvar\int_{\bR} \varphi_{\sigma^2t}(r-x)B(x)\,dx, 
 \end{aligned}  \label{Zint}\ee
where $\varphi_{\nu^2}(x)=(2\pi\nu^2)^{-1/2}\exp(-x^2/2\nu^2)$
denotes the centered Gaussian density with variance $\nu^2$,
and $\Phi_{\nu^2}(x)=\int_{-\infty}^{x}\varphi_{\nu^2}(y)dy$
the distribution function. 

 $\{Z(t,r): t\in \bR_{+},r\in \bR\}$ is a mean zero  
Gaussian process. Its  covariance can be expressed as follows:
with 
\be \Psi_{\nu^2}(x)=\nu^2\varphi_{\nu^2}(x)
-x\big(1-\Phi_{\nu^2}(x)\big)\label{Psi}\ee
define two covariance functions on $(\bR_+ \times \bR)\times(\bR_+ \times \bR)$ by 
 \be \Gamma_1\big((s,q),(t,r)\big)=\Psi_{\sigma^2 (t+s)}(r-q)
 - \Psi_{\sigma^2 \vert t-s\vert}(r-q)  \label{Ga1}\ee
and 
\be \Gamma_2\big((s,q),(t,r)\big)=\Psi_{\sigma^2 s}(-q)+\Psi_{\sigma^2 t}(r)-\Psi_{\sigma^2 (t+s)}(r-q).  \label{Ga2}\ee
Then 
\be \mE[Z(s,q)Z(t,r)]=\rho_0\Gamma_1\big((s,q),(t,r)\big)+
\qvar^2\Gamma_2\big((s,q),(t,r)\big). \label{Zcov}\ee
(Boldface $\mP$ and $\mE$ denote generic probabilities and expectations  
not connected with the RWRE model.)  

The theorem we state is for the finite-dimensional distributions of the current
process, scaled and  centered by its quenched mean:  
\[ \overline{Y}_n(t,r)=n^{-1/4}\bigl\{ Y_n(t,r)-E^\om[Y_n(t,r)]\bigr\}. \]
Fix any $N\in\bN$,   time points  
$0<t_1<t_2<\cdots<t_N \in \bR_{+}$,  
  space points  $r_1,r_2,\dotsc,r_N \in \bR$ and an 
  $N$-vector $\thvec=(\theta_1,\dotsc, \theta_N)\in\bR^N$.  Form the linear combinations 
\[ \overline Y_n(\thvec)= \sum_{i=1}^N \theta_i \overline Y_n(t_i,r_i)
\quad\text{and}\quad  Z(\thvec)= \sum_{i=1}^N \theta_i  Z(t_i,r_i). \]
 
\begin{theorem}
\label{fddthm}  Consider independent particles on $\bZ$ in an i.i.d.~space-time environment 
with bounded jumps, under assumptions {\rm\eqref{Pass}} and {\rm\eqref{Rass}}.
 Let the $\om$-dependent initial distribution satisfy {\rm\eqref{Dass}}  
  and {\rm\eqref{momass}}.  With definitions as above, quenched characteristic functions converge
in $L^1(\P)$: 
\be  \lim_{n\to\infty} \E\bigl\lvert  E^\om(e^{i \overline Y_n(\thvec)}) - \mE(e^{i Z(\thvec)}) \bigr\rvert 
 =0.  \label{chflim}\ee 
  In particular,  under the averaged 
 distribution $P$,   convergence  in distribution  holds for 
   the $\bR^N$-valued  vectors as $n\to\infty$:   
 \[ \big(\overline Y_n(t_1,r_1), \overline Y_n(t_2,r_2),\cdots,\overline Y_n(t_N,r_N)\big)
 \Rightarrow \big(Z(t_1,r_1),Z(t_2,r_2),\cdots, 
Z(t_N,r_N)\big).\]
\end{theorem}

While we do not have a quenched limit (convergence of distributions under a fixed $\om$),
limit \eqref{chflim} does imply that, if a quenched limit exists, it is the same as we have
found.  

A special case of the above theorem is the stationary situation. The proof of the following corollary comes by a direct computation using \eqref{Zcov}.

\begin{corollary}
Consider the same setting as in the previous theorem. 
If furthermore variables $\eta_0$ have conditional distribution \eqref{def:muom}, and more generally when $\sigma_0^2=\rho_0$, process $Z(t,0)$ has covariance
	\[\mE[Z(s,0)Z(t,0)]=\frac{\rho_0\sigma}{\sqrt{2\pi}}(\sqrt s+\sqrt t-\sqrt{\abs{t-s}}\,),\]
i.e.\  $\rho_0^{-1}\sigma^{-1}\sqrt{\pi/2}\,Z(t,0)$ is a fractional Brownian motion with Hurst parameter $1/4$.
\end{corollary}

\medskip



The next two theorems are on fluctuations of the quenched mean process $E^\w Y_n(t,r)$ in the special case of 
one-dimensional random walks with admissible steps 0 and 1. Although we expect the result to hold for more general random walks, this is the only case for which we are able to characterize the fluctuations.  Let $\sigma^2_D=\Vvv(\w_{0,0})$ and $\alpha=\bE\w_{0,0}(1-\w_{0,0})$. Note that $v=p(1)=\E\w_{0,0}$ and $\sigma^2=v(1-v)$.\medskip

First, we consider the case of an initial configuration $\eta_0$ with independent quenched means.

\begin{theorem}\label{th:SHE}
Let $\{\eta_0(x):x\in\Z\}$ be such that the quenched means $\{E^\w\eta_0(x):x\in\Z\}$ are independent with mean $\rho_0$ and variance $\sigma_0^2$. Assume that there exists $\e>0$ such that
$\sup_x\bE[\abs{E^\w \eta_0(x)}^{2+\e}]<\infty$. Assume the $\eta_0$-variables are independent of the transition probabilities $\{\w_{x,t}:(x,t)\in\Z\times\Z_+\}$.
Then the finite-dimensional marginals of the process 
$\{n^{-1/4}E^\w \bigl(Y_n(t,r)-\rho_0 r\sqrt n\,\bigr):t\ge0,r\in\R\}$ converge weakly as $n\to\infty$ to those of the unique mild solution to the stochastic heat equation
	\[z_t=\frac{\sigma^2}2 z_{rr}+\frac{\rho_0\sigma_D}{\sqrt{\alpha}}\,{\dot W},\quad z(0,r)=\sigma_0 B(r),\]
where $\dot W$ is space-time white noise and $B$ a two-sided Brownian, independent of $\dot W$.
\end{theorem}

The above includes the case when $\eta_0$ is independent of $\w$ altogether. In that case, $\sigma_0=0$ and thus the initial condition becomes $z(0,r)\equiv0$.

Next, we look at the stationary case. The reason this is different from the previous theorem is that now the quenched means of the initial occupation variables are not independent.

\begin{theorem} \label{th:fBM}
Let $\{\eta_0(x):x\in\Z^d\}$ be distributed according to \eqref{def:muom} with $\rho=1$. 
Assume the averaged probabilities $p_0=p_1=1/2$ so that $v=1/2$. 
Then
for $t\ge s>0$ we have
\[ \lim_{n\to\infty} \frac1{\sqrt{n}} \Cvv\bigl(E^\w Y_n(s,0),E^\w Y_n(t,0)\bigr)=   
 \frac{1}{2\sqrt{2\pi}}\bigl( \tfrac14 \alpha^{-1}-1\bigr)(\sqrt t+\sqrt s-\sqrt{t-s}\,). \]  
\end{theorem}

The above limit matches the covariance structure of a constant ($( \tfrac14 \alpha^{-1}-1)^{1/2}/(2\pi)^{1/4}$) times a fractional Brownian motion with Hurst parameter $1/4$.
Theorems \ref{th:SHE} and \ref{th:fBM} are proved in Section \ref{var-covar}. At the end of that section, we explain why we expect the same limiting behavior in the setting
of Theorem \ref{th:fBM} as that of Theorem \ref{th:SHE} and how this would imply the fractional Brownian motion limit.

\subsection*{Further notational conventions} 
  $\bN=\{1,2,3,\dotsc\}$ and $\bZ_+=\{0,1,2,\dotsc\}$. 
 Multistep 
transition probabilities from time $s$ to time $t>s+1$ are 
\[\pi_{s,t}^\w(x,y)=\sum_{u_1,\dotsc,u_{t-s-1}\in\bZ^d}
\pi_{s,s+1}^\w(x,u_1)\pi_{s+1,s+2}^\w(u_1,u_2)\dotsm
\pi_{t-1,t}^\w(u_{t-s-1},y).\] 
We omit floor notation from time parameters, and so for the walk, $X_t=X_{\fl{t}}$ for real $t\ge 0$.  No jumps happen between integer times.  

 $\mathcal{S}$ denotes the set of all measures $\mu$ on $(\mathbb{Z}_{+})^{\mathbb{Z}^d}$ that  are invariant under spatial translations.   $\mathcal{S}_e$ denotes the subset of $\mathcal{S}$ consisting of ergodic measures.   $\mathcal{I}$ denotes the set of measures that  are invariant for the particle evolution, that is $\mu_t=\mu S(t)=\mu \mbox{  for all } t\in \bZ_{+}$ ($\mu_t$ and $\mu S(t)$ here denote the measure on configurations at time $t$ when the initial measure on configurations is $\mu$).
$\bE$, $E^{\om},E,E_{\eta},\mathbf{E}$ etc will denote expectations with respect to $\bP$, $P^{\om},P,P_{\eta},\mathbf{P}$, etc. Variances and covariances are denoted similarly. Constants $C$ can change from term to term.

\section{Coupled Process} \label{coupsec} 

This section describes the coupling that will be used to prove Theorem \ref{invthm}.
We couple two processes $\eta_t$ and $\zeta_t$ so that matched particles move
together  forever, while unmatched particles move independently.  To do this precisely,
choose for each space-time point $(x,t)$ a collection  
$\Xi_{x,t}=\{\jump_{x,t}^{0,j}, \jump_{x,t}^{+,j}, \jump_{x,t}^{-,j} :j\in\bN\}$ 
of  i.i.d.~$\bZ^d$-valued jump vectors from distribution 
$\om_{x,t}$.  Given initial configurations   $\eta_0$ and $\zeta_0$,  
perform the following actions. At each site $x$  set 
\[ \text{ $\xi_0(x)=\eta_0(x)\wedge\zeta_0(x)$, 
\  $\beta^{+}_0(x)=(\eta_0(x)-\zeta_0(x))^{+}$, and  $\beta^{-}_0(x)=(\eta_0(x)-\zeta_0(x))^{-}$.}  \]
$\xi_0(x)$ is the number of matched particles,  while $\beta^\pm_0(x)$ count
the unmatched $(+)$  and $(-)$ particles. 
Move particles  from each site $x$ as follows: the $\xi_0(x)$ matched particles  jump to locations
 $x+ \jump_{x,0}^{0,j}$ for $j=1,\dotsc,\xi_0(x)$, the $\beta^+_0(x)$
  $(+)$ particles jump to  locations
 $x+ \jump_{x,0}^{+,j}$ for $j=1,\dotsc,\beta^+_0(x)$,  and the $\beta^-_0(x)$
  $(-)$ particles jump to  locations
 $x+ \jump_{x,0}^{-,j}$ for $j=1,\dotsc,\beta^-_0(x)$. 
 After all jumps from all sites  have been executed, 
  match as many pairs of  $(+)$ and $(-)$
 particles   at the same site as possible.   This means that 
 a $(+-)$ pair together at the same site 
 merges to create a single $\xi$-particle at the same site.   (For example, if after the jumps site $y$ contains $s$  $\xi$-particles, $k$  $(+)$ particles and $\ell$  $(-)$
 particles, then set  $\xi_1(y)=s+k\wedge \ell$ and   $\beta^\pm_1(y)=(k-\ell)^\pm$.)     Since  particles are not labeled, 
 it is immaterial which particular $(+)$ particle merges with a particular $(-)$ particle. 
 When this is complete we have defined the state 
 $(\xi_1(x), \beta^{+}_1(x), \beta^{-}_1(x))_{x\in\bZ^d}$ at time $t=1$. 
 Then repeat, utilizing
 the jump variables for time  $t=1$.  And so on.   
 
 This produces a joint process $(\xi_t,\beta^+_t, \beta^-_t)$ such that  
 \[ \text{ $\xi_t(x)=\eta_t(x)\wedge\zeta_t(x)$, 
\  $\beta^{+}_t(x)=(\eta_t(x)-\zeta_t(x))^{+}$, and  $\beta^{-}_t(x)=(\eta_t(x)-\zeta_t(x))^{-}$.}  \]
The $\eta$ and $\zeta$ processes are recovered from 
\[   \eta_t(x)= \xi_t(x)+\beta^{+}_t(x)  
\quad\text{and}\quad 
\zeta_t(x)= \xi_t(x)+\beta^{-}_t(x). \]

  The definition has the effect that a matched pair
of $\eta$ and $\zeta$ particles stays forever together, while a pair of $(+)$ and $(-)$ particles
together at a site annihilate each other and turn into a matched pair.  If we
are only interested in the evolution of the discrepancies $(\beta^+_t,\beta^-_t)$
we can discard all matched pairs as soon as they arise, and simply consider 
independently evolving  $(+)$ and $(-)$ particles that annihilate each other upon meeting. 

If we denote by $\Xi=\{ \Xi_{x,t}: x\in\bZ^d,t\in\bZ\}$ 
 the collection of jump
variables,  and by $G_{0,t}$ the function that constructs the values at the 
origin at time $t$:
\[  \bigl(\xi_t(0),\beta^+_t(0), \beta^-_t(0)\bigr)= G_{0,t}(\eta_0,\zeta_0,\Xi)  \]
then it is clear that the values at other sites $x$ are constructed by 
applying this same function to shifted input: 
\be  \bigl(\xi_t(x),\beta^+_t(x), \beta^-_t(x)\bigr)= G_{0,t}(\shift_x\eta_0,\shift_x\zeta_0,\shift_x\Xi).
\label{shift2}\ee 
Here $\shift_x$ is a spatial shift: $(\shift_x\eta)(y)=\eta(x+y)$ and 
$(\shift_x\Xi)_{y,t}=\Xi_{x+y,t}$ for $x,y\in\bZ^d$.  
In particular, if the initial
distribution $\wt\mu$ of the pair $(\eta_0,\zeta_0)$ is invariant  and ergodic
under the shifts $\shift_x$,  while  $\{ \Xi_{x,t}: x\in\bZ^d,t\in\bZ\}$  are i.i.d.\ and independent
of $(\eta_0,\zeta_0)$, it follows first that the triple $(\eta_0,\zeta_0,\Xi)$
is ergodic, 
 and then from \eqref{shift2} that for each fixed $t$  the configuration 
$(\xi_t,\beta^+_t, \beta^-_t)$ is invariant and ergodic under the shifts $\shift_x$. 

Let $\wt{\cS}$, resp.\ $\wt{\cS}_e$, 
denote the set of spatially invariant, resp.\ ergodic, probability distributions on 
   pairs $(\eta,\zeta)$ of configurations of occupation variables. 
   
\begin{lemma} 
\label{betadecrease} Let $\wt{\mu} \in \wt\cS$.
The expectations  $E_{\wt{\mu}}[\beta_t^{+}(x)]$ and 
 $E_{\wt{\mu}}[\beta_t^{-}(x)]$ are independent of $x$
and nonincreasing in $t$.  
\end{lemma}
 \begin{proof}  
The independence of $x$ is due to the shift-invariance from
\eqref{shift2}.  That $E_{\wt{\mu}}[\beta_t^{\pm}(x)]$ is nonincreasing in $t$ follows from the fact that discrepancy particles are not created,
only annihilated.
 \end{proof}
 

\begin{proposition}
\label{rho1gerho2} Let $d=1 \mbox{ or } 2$.
Suppose $\wt{\mu}\in \wt\cS_e$. Let $E_{\wt{\mu}}[\eta(0)]=\rho_1$ and $E_{\wt{\mu}}[\zeta(0)]=\rho_2$. If $\rho_1\ge \rho_2$, we have
\[ E_{\wt\mu}[\beta^-_t(0)]=E_{\wt{\mu} }[(\eta_t(0)-\zeta_t(0))^{-}] \to 0 \mbox{ as } t \to \infty.\]
\end{proposition}
\begin{proof}
We already know from Lemma \ref{betadecrease} that $E_{\wt{\mu}}[\beta_{t}^{\pm}(0)]$ 
cannot increase. To get a contradiction
let us assume that  $E_{\wt{\mu}}[\beta_{t}^{-}(0)] \ge \delta$ for all $t$ and some $\delta >0$. Since  $E_{\wt{\mu}}[\beta_{t}^{+}(0)]- E_{\wt{\mu}}[\beta_{t}^{-}(0)]=E_{\wt{\mu}}[\eta_{t}(0)]- E_{\wt{\mu}}[\zeta_{t}(0)]=\rho_1-\rho_2\ge 0$, we also have $E_{\wt{\mu}}[\beta_{t}^{+}(0)] \ge \delta$ for all $t$.

At time $0$, assign  labels separately to the $(+)$ and $(-)$ particles from some countable
label sets $\cJ^+$ and $\cJ^-$  and denote the locations of these particles by 
  $\{w_i^{+}(t),w_j^{-}(t): i\in\cJ^+, j\in \cJ^-\}$. Each $(+)$ and $(-)$ particle retains its label throughout its lifetime. 
The lifetime of $(+)$ particle $j \in \cJ^+$
  is  
\[
 \tau_j^{+} = \inf\{ t\ge 0:w_j^{+} \mbox{ is annihilated by a $(-)$ particle} \}. 
 \] 
If $ \tau_j^{+}=\infty$, then $j$ is {\sl immortal}. Similarly define $\tau_j^{-}$. Let 
\[ \beta^{\pm}_{0,t}(x)=\sum_j \ind\lbrace w_j^{\pm}(0)=x, \tau_j^{\pm}>t\rbrace\]
denote the number of $(\pm)$ particles initially at site $x$ that live past time $t$. 
We would like to claim  that  for a fixed $t$ the configuration 
  $\{(\beta^{+}_{0,t}(x),\beta^{-}_{0,t}(x)): x\in \bZ^d\}$ is   invariant and   ergodic under the 
  spatial shifts $\shift_x$.
This will be true if the evolution is given by   a mapping $F_{0,t}$ so that 
$(\beta^{+}_{0,t}(x),\beta^{-}_{0,t}(x))= F_{0,t}(\shift_x\eta_0,\shift_x\zeta_0,\shift_x\Xi)$
 for all $x\in\bZ^d$.   Such a mapping can be created by specifying precise rules
 for the movement and annihilation of  $(+)$ and $(-)$ particles that are naturally invariant
 under shifts.    For example, we can take $\cJ^\pm\subset\bZ$ and give the sites of $\bZ^d$ some ordering.  Label particles initially in increasing order,   so that $i<j$ implies $w_i^{\pm}(0)\le w_j^{\pm}(0)$.    Then at each time step  particles   from a given site   are   distributed to their subsequent  locations in increasing order, and $(+,-)$ pairs are matched beginning with lowest labels.  Of course the overall ordering of particles is not preserved, but this mechanism does not depend on the absolute labels, only their ordering,  and respects the spatial translations.

Then the  ergodic theorem implies that 
\[ E_{\wt{\mu}}[\beta^{\pm}_{0,t}(0)]= \lim_{n \to \infty} \frac{1}{(2n+1)^d}\sum_{\abs{x}\le n} \beta^{\pm}_{0,t}(x) \mbox{ a.s. }\]
Here $\abs{x}$ is the $\ell^\infty$ norm:  for a vector
$x=(x_1,\dotsc, x_d)$, 
$\abs{x}=\max_{1\le i\le d}\abs{x_i}$. 
Since particles take jumps of magnitude at most $R$, 
\begin{align*}
\delta\le E_{\wt{\mu}}[ \beta^{\pm}_{t}(0) ] 
&=\lim_{n \to \infty} \frac{1}{(2n+1)^d}\sum_{\abs{x}\le n} 
\beta^{\pm}_{t}(x) \\
&\le \lim_{n \to \infty} \frac{1}{(2n+1)^d}\sum_{\abs{x}\le n+Rt} 
\beta^{\pm}_{0,t}(x) =  E_{\wt{\mu}}[\beta^{\pm}_{0,t}(0)]. 
\end{align*}

The initial occupation numbers of immortal $+/-$ particles are  
\[ \beta^{\pm}_{0,\infty}(x)=\lim \limits_{t \to \infty}\beta^{\pm}_{0,t}(x).\]
The limit exists by monotonicity. 
This limit produces again a functional relationship of the type \eqref{shift2}:
\begin{align*}
  \big(\beta^{+}_{0,\infty}(x),\beta^{-}_{0,\infty}(x)\big)
&=\lim \limits_{t \to \infty}  \big(\beta^{+}_{0,t}(x),\beta^{-}_{0,t}(x)\big)=\lim \limits_{t\to \infty} 
F_{0,t}(\shift_x\eta_0,\shift_x\zeta_0,\shift_x\Xi)\\
&= F_{0,\infty}(\shift_x\eta_0,\shift_x\zeta_0,\shift_x\Xi).  
\end{align*}
Thereby 
  $\{(\beta^{+}_{0,\infty}(x),\beta^{-}_{0,\infty}(x)): x\in \bZ^d\}$ is spatially invariant and ergodic. 

By the ergodic theorem again 
\[ E_{\wt{\mu}}[\beta^{\pm}_{0,\infty}(0)]
= \lim_{n \to \infty} \frac{1}{(2n+1)^d}
\sum_{\abs{x}\le n} \beta^{\pm}_{0,\infty}(x) 
\quad\mbox{a.s. }\]
while by the monotone convergence theorem 
\[  E_{\wt{\mu}}[\beta^{\pm}_{0,\infty}(0)]=\lim\limits_{t \to \infty}  E_{\wt{\mu}}[\beta^{\pm}_{0,t}(0)] \ge \delta.\]

We have shown that the assumption  $E_{\wt{\mu}}[\beta_{t}^{-}(0)] \ge \delta$ leads to the existence of positive densities of immortal 
$(+)$ and $(-)$ particles.  However, a situation like this will never arise for $d=1\mbox{ or } 2$, the reason being that any two particles on the lattice will meet each other infinitely often. More precisely, fix any two particles
 and let
 $X^+_\abullet$ and $X^-_\abullet$  denote the walks 
undertaken by these two particles.  Then  $X^+_\abullet$ and $X^-_\abullet$ 
are two independent walks in a common environment $\om$.
Let  $Y_t=X^+_t-X^-_t$. 
If we average out the environment, then $Y_t$ is a Markov
chain on $\bZ^d$ with transition $q(x,y)$ given by \eqref{qxy}.  Away from the origin this is a symmetric random walk with bounded steps, and hence   recurrent when $d=1$ or $2$. Thus $Y_t=0$ infinitely often. We have   arrived at a contradiction and the proposition is proved.\end{proof}
 

\bigskip

\section{Invariant measures} \label{inv}

In this section we prove Theorem \ref{invthm}. 
We begin by deriving the well-known invariant density for 
the environment process seen by a single tagged particle. 

\begin{proposition} There exists a function 
$0\le f<\infty$ on $\Omega$ such that 
$\E f=1$, $\E(f^2)<\infty$, 
$f(\om)$ is a function of $\wlev_{-\infty,-1}$, and  
\be
f(\om)=\sum_{x\in\bZ^d} 
f(T_{x,-1}\om)\pi_{-1,0}^\om(x,0)\quad\text{$\P$-almost surely.} 
\label{finv}\ee
\label{finvprop}\end{proposition}

\begin{proof}
For $N\in\bZ_+$  define
\be
f_N(\om)= \sum_{z\in\bZ^d} \pi_{-N,0}^\om(z,0).
\label{def:fN}\ee
$f_N(\om)$ is $\kS_{-N,-1}$-measurable and a martingale
with $\E f_N=1$. By the martingale convergence theorem
we can define
\[
f(\om)=\lim_{N\to\infty}f_N(\om)
\quad\text{($\P$-almost sure limit).}\]
Property \eqref{finv} follows because all the sums
involved are finite:
\begin{align*}
&\sum_x f(T_{x,-1}\om)\pi_{-1,0}^\om(x,0)
=\lim_{N\to\infty} \sum_x f_N(T_{x,-1}\om)\pi_{-1,0}^\om(x,0)\\
&=\lim_{N\to\infty} \sum_{z,x} \pi_{-N-1,-1}^\om(z,x)\pi_{-1,0}^\om(x,0)
=\lim_{N\to\infty} \sum_{z} \pi_{-N-1,0}^\om(z,0) \\
&=\lim_{N\to\infty} f_{N+1}(\om)=f(\om).
\end{align*}
In Lemma \ref{lm:fNmoment2} below  we show the $L^2$ boundedness 
of the sequence $\{f_N\}$. 
This implies that $f_N\to f$ also in $L^2$ and thereby
implies the remaining statements $\E f=1$ and $\E(f^2)<\infty$.  
\end{proof}

The following addresses the positivity of $f$.

\begin{lemma}
$\P(f>0)=1$ if and only if there exists an $x$ such that $\P\{\pi_{0,1}(0,x)>0\} = 1$.
\end{lemma}

\begin{proof}
If there does not exist an $x$ as in the claim, then by independence of the environment and the finite step-size assumption we see that 
	\[\P\{\forall x: \pi_{-1,0}(x,0)=0 \} >0.\]� �
But then   \eqref{finv} implies that $\P(f=0)>0$.
Conversely, if there exists an $x$ as in the claim, then \eqref{finv} implies that if $f(\w)=0$ then $f(T_{x,-1}\w)=0$.� Shift-invariance implies the two events are in fact equal, almost surely.
This in turn implies that $\{f=0\}$ is a trivial event and since $\E[f]=1$ we have that $f>0$ a.s.
\end{proof}

To prove the $L^2$ estimate for $f_N$ we develop a Green
function bound for the Markov 
chain defined as the difference of two walks. 
Let $X^x_t$ and $\widetilde{X}^y_t$ be two independent walks in a
common environment $\om$, started at 
$x, y\in\bZ^d$, and $Y_t=X^x_t-\widetilde{X}^y_t$. 
Under the  averaged measure 
$Y_t$ is a Markov chain on $\Zb^d$ with transition probabilities 
$q(x,y)$ defined by \eqref{qxy}.  
$Y_t$ can be thought of 
as  a symmetric random walk on $\Zb^d$ whose transition
has been perturbed at the origin. The corresponding  homogeneous,
unperturbed random walk is $\Yhom_t$ with 
transition probability $\qhom$ in \eqref{qbarxy}.   Write
$P_x$ and $\Phom_x$ for the path probabilities of $Y_\abullet$ and $\Yhom_\abullet$.
Define hitting times of $0$ for both walks $Y_t$ and $\Yhom_t$ by 
\be\tau=\inf\{n\ge 1: Y_n=0\}
\quad\text{and}\quad 
\tauhom=\inf\{n\ge 1: \Yhom_n=0\}. 
\label{tau}\ee
Denote the $k$-step transition probabilities by $q^k(x,y)$ and $\qhom^k(x,y)$.

\begin{lemma}
There exists a constant $C<\infty$ such that
 for all $x \in {\bZ}^d$ and $N \in \bN$, 
\[  \sum_{k=0}^{N} q^k(x,0) \le C\sum_{k=0}^{N} \qhom^k(x,0).\]
\label{lm:GGbar} \end{lemma}
\begin{proof}  Suppose we had the bound for $x=0$. 
Then it follows  for $x\ne 0$: 
\begin{align*}
 \sum_{k=0}^{N} q^k(x,0) &= E_x \Big[\sum_{k=0}^N \ind_{\lbrace Y_k=0\rbrace}\Big]
 = E_x\Big[\sum_{i=0}^N \ind_{\{\tau=i\}}\sum_{k=i}^N \ind_{\{Y_k=0\}}\Big]\\
 &= \sum_{i=0}^N P_x( \tau=i) \sum_{k=0}^{n-i} q^k(0,0) 
 \le  C \sum_{i=0}^N \Phom_x( \tauhom=i) 
\sum_{k=0}^{n-i} \qhom^k(0,0) \\
 &= C \sum_{k=0}^N\qhom^k(x,0)
\end{align*}

It remains to prove the result for  $x=0$. 
Let $\sigma_0=0$ and 
\[
\sigma_{j+1}=\inf\{n>\sigma_j: \text{ $Y_n=0$ and
$Y_k\neq 0$ for some $k\in\{\sigma_j+1, \dotsc,n-1\}$}\}.
\]
These are the successive times of arrivals to $0$ 
following excursions away from 
$0$. 
Let $W_j$, $j\geq 0$, be the durations of the sojourns at $0$,
in other words
\[
Y_n=0 \text{ iff }\sigma_j\leq n<\sigma_j+W_j
\text{ for some $j\geq 0$.}
\]
Sojourns are geometric and independent of the past,
 so on the event $\{\sigma_j<\infty\}$,
\[
E_0(W_j \,\vert\, \cF^Y_{\sigma_j}) = \frac1{1-q(0,0)}.
\]
Let $J_N=\max\{j\geq 0: \sigma_j\leq N\}$ mark the
  last sojourn at $0$ that started by time $N$.
Then 
\begin{align*}
E_0\Bigl[\, \sum_{k=0}^N \ind\{Y_k=0\}\Bigr] &\leq
E_0\Bigl[\, \sum_{j=0}^{J_N} W_j \Bigr]
= \sum_{j=0}^\infty E_0\bigl[\, \ind\{\sigma_j\leq N\} W_j\bigr] \nn\\
&= \frac1{1-q(0,0)} E_0(1+J_N).
\end{align*}
Assumption \eqref{ass:q00} guarantees that $q(0,0)<1$.  

It remains to bound $E_0(1+J_N)$ in terms of 
$\sum_{k=0}^N \qhom^k(0,0)$.  The key is that once the Markov
chain $Y_k$ has left the origin, it follows the same transitions as
 the homogeneous walk  $\Yhom_k$ until the next visit to $0$.
For $z\ne 0$ let 
\[ K^{z}_N=\inf \{k\ge 1 :T^{z}_1+T^{z}_2+\cdots+ T^{z}_k \ge N\}\]
 where the $\{T^{z}_i\}$ are i.i.d.\ with common distribution 
$\Phom_z\{\tauhom\in\cdot\,\}$. Imagine constructing the path $Y_k$ 
so that every step away from $0$ is followed by an excursion
of $\Yhom_k$ that ends at $0$ (or continues forever if $0$ is 
never reached). The step bound \eqref{step-bd} implies that 
 $P_0\{\abs{Y_1}\le 2R\}=1$.  Then  there is stochastic 
dominance that gives 
\be  E_0(J_N) \le  \sum_{z\ne 0\,:\, \abs{z}\le 2R} \Ehom(K^z_N).
\label{JNK}\ee

By  T32.1 in Spitzer \cite[p.~378]{spitzer}, for $z\ne 0$
\be \label{a} \lim_{n \to \infty} 
\frac{\Phom_z(\tauhom>n)}{\Phom_0(\tauhom>n)} =\ahom(z) \ee
where the potential kernel $\ahom$ is 
\be \ahom(z)=\lim_{n \to \infty} \Bigl\{\; \sum_{k=0}^n\qhom^k(0,0)
-\sum_{k=0}^n\qhom^k(z,0)\Bigr\}. \label{defbara}\ee
By  P30.2 in \cite[p.~361]{spitzer}  $\ahom(z)>0$ for all $z \ne 0$
for $d=1,2$. 
 For $d\ge 3$ by transience
\begin{align*} 
\ahom(z)=\sum_{k=0}^{\infty}\qhom^k(0,0)-
\sum_{k=0}^{\infty}\qhom^k(z,0)
= (1-\Fhom(z,0))\sum_{k=0}^{\infty}\qhom^k(0,0)>0
\end{align*}
where $\Fhom(z,0)=\Phom_z\{\,\text{$\Yhom_n =0$ 
for some $n\ge 1$}\}<1.$ 

From \eqref{a} and $\ahom(z)>0$,  
there exist $0<c(z),C(z)<\infty$ such that for all $n$,
\[ c(z) \Phom_0(\tauhom>n) \le \Phom(T^{z}_1 >n) 
\le C(z) \Phom_0(\tauhom>n) \]
and hence
\[  c(z)\Ehom_0\big[\tauhom \wedge N\big] \le \Ehom\big[T^{z}_1 \wedge N\big]\le C(z) \Ehom_0\big[\tauhom \wedge N\big]. \]

By Wald's identity and some simple bounds
(see Exercice 4.4.1 in \cite[Sect.~4.4]{durr})
\[ \frac{N}{\Ehom\big(T^{z}_1 \wedge N\big)} \le \Ehom\big[K^{z}_N \big] 
\le \frac{2N}{\Ehom\big(T^{z}_1 \wedge N\big)}. \]
Let $\{\tauhom_i\}$ be i.i.d.~copies of $\tauhom$ from \eqref{tau} 
and put 
 \[ M_N =\inf\{ k\ge 1:\tauhom_1+\tauhom_2+\cdots+\tauhom_k \ge N \}. \]
Then we have a similar relation: 
\[ \frac{N}{\Ehom_0\big(\tauhom \wedge N\big)} \le \Ehom_0\big[M_N \big] \le \frac{2N}{\Ehom_0\big(\tauhom \wedge N\big)} \]
Combining the above lines: 
\be \Ehom\big[K^{z}_N\big] \le \frac{2N}
{ \Ehom\big[T^{z}_1 \wedge N\big]} 
\le  \frac{2N}{c(z) \Ehom_0\big[\tauhom \wedge N\big]}\le 
\frac{2}{c(z)} \Ehom_0\big[M_N\big]. \label{K2}\ee
Considering excursions of the $\Yhom$-walk away from $0$,
 \be \Ehom_0\big[M_N\big] \le 
1+\sum_{k=0}^N \qhom^k(0,0) \le 2\sum_{k=0}^N\qhom^k(0,0).
\label{MN}\ee 
The proof is now complete 
with a combination of \eqref{JNK}, \eqref{K2} and \eqref{MN}.
\end{proof}

From the previous lemma follows  the $L^2$ estimate for $f_N$
which completes the proof of Proposition 
\ref{finvprop}.  

\begin{lemma} There exists a 
constant $C<\infty$ such that $\E(f_N^2)\le C$ for all $N$.
\label{lm:fNmoment2}
\end{lemma} 

\begin{proof}
By translations 
\begin{align*}
\E(f_N^2)&= \sum_{x,z}\E\pi^\om_{-N,0}(x,0)\pi^\om_{-N,0}(z,0)\\
&=\sum_y\E P^\om\{X^y_N=\widetilde{X}^0_N\}
=\sum_y q^N(y,0).
\end{align*}
By the submartingale property
 $\E(f_N^2)$ is nondecreasing in $N$. Hence it  suffices
to show the existence of a constant $C$ such  that 
\be
\sum_{k=0}^N \E(f_k^2)\le C(N+1) \qquad\text{for all $N$.}
\label{intgoal8}\ee
From above, by Lemma \ref{lm:GGbar}  and the spatial homogeneity
of the $\Yhom$-walk, 
\begin{align*} 
\sum_{k=0}^N \E(f_k^2) &= 
 \sum_{k=0}^{N} \sum_x q^k(x,0) 
\le C\sum_{k=0}^{N} \sum_x  \qhom^k(x,0)\\
&=C\sum_{k=0}^{N} \sum_x  \qhom^k(0,x) = C(N+1).
\qedhere  \end{align*}
\end{proof}

Property \eqref{finv} implies that the probability measure 
$\P_\infty(d\om)$ $=$ $f(\om)\,\P(d\om)$ is invariant for the 
process $T_{X_n,n}\om$. 
Recall from \eqref{def:muom}  the product measure 
\be
\mu^{\rho,\om}(d\eta)=\bigotimes_{x\in\bZ^d} 
\pss^{\rho f(T_{x,0}\om)}\big(d\eta(x)\big)
\label{def:muom2}\ee
where $\pss^\lambda$ is  Poisson($\lambda$)  
 distribution.  By the definition of $f$,  
 $\mu^{\rho,\om}$ depends on $\om$ only through the levels
$\wlev_{-\infty,-1}$. 

\begin{lemma} The following holds for $\P$-a.e.~$\om$. 
Let $\eta_0$ be $\mu^{\rho,\om}$-distributed. Then
for all times $t\in\bZ_+$, under 
the evolution in the environment $\om$, $\eta_t$ is
$\mu^{\rho,T_{0,t}\om}$-distributed, and in particular independent
of the environment $\wlev_t$ at level $t$. 
\label{lm:inv1}\end{lemma}

\begin{proof} Consider the evolution under a fixed $\om$.
 The claim made in the lemma is true at time $t=0$ by the
construction. Suppose it is true up to time $t-1$. Then over
$x\in\bZ^d$ 
the variables $\eta_{t-1}(x)$ are independent
Poisson variables with means $\rho f(T_{x,t-1}\om)$.  
Each particle at site $x$ chooses its
next position $y$ independently with probabilities
$\pi_{t-1,t}^\om(x,y)$. As with marking  a Poisson process
with independent coin flips, the consequence is that 
 the numbers of particles going
from $x$ to $y$ are independent Poisson variables with means
$\rho f(T_{x,t-1}\om)\pi_{t-1,t}^\om(x,y)$, over
all pairs $(x,y)$.  Since sums of independent
Poissons are Poisson, the variables
$(\eta_t(y))_{y\in\bZ}$ are again independent Poissons and
$\eta_t(y)$ has mean
\begin{align*}
\sum_x \rho f(T_{x,t-1}\om)\pi_{t-1,t}^\om(x,y)
&=\sum_z \rho f(T_{z,-1}T_{y,t}\om)\pi_{-1,0}^{T_{y,t}\om}(z,0)\\
&=\rho f(T_{y,t}\om).
\end{align*}
The last equality is from \eqref{finv}.  

We have shown that $\eta_t=(\eta_t(y))_{y\in\bZ^d}$ has 
distribution $\mu^{\rho,\,T_{0,t}\om}$. This measure
is a function of $\wlev_{-\infty,t-1}$, hence  independent of $\wlev_t$ under
$\P$.  
\end{proof} 

Recall from \eqref{defmuu}   the averaged measure 
$
\mu^{\rho}=\int \mu^{\rho,\om}\,\P(d\om).
$

\begin{lemma}
The measure $\mu^\rho$ is invariant and ergodic under
spatial shifts $\shift_x$.    The tail $\sigma$-field of
the state space $\bZ_+^{\bZ^d}$ is trivial under $\mu^\rho$. 
\label{erglm}\end{lemma}
\begin{proof}
Invariance under $\theta_x$ comes from 
$\mu^{\rho,\om}\circ\shift_x^{-1}=\mu^{\rho,T_{x,0}\om}$ 
and the invariance of $\bP$. 
Ergodicity will follow from tail triviality. 
 
Let  $B\subseteq\bZ_+^{\bZ^d}$ be  
a tail event.   Then by Kolmogorov's 0-1 law $\mu^{\rho,\om}(B)\in\{0,1\}$ for each $\om$.
We need to show that $\mu^{\rho,\om}(B)$ is $\bP$-a.s.\ constant.
For this it suffices to show that 
  $ \mu^{\rho,\om}(B)$ is itself  (almost surely) a tail measurable function of $\om$. 

 Consider a  ball $\Lambda=\{(z,s): \abs{s}+\abs{z}\le M\}$ in the space-time lattice
 $\bZ^{d}\times\Z$.  Since the step size of the
walks is bounded by $R$,   for each $x\in\Z^d$ and $N\ge 1$
\[  f_N(T_{x,0}\om)=\sum_{z\in\bZ^d} \pi_{-N,0}^\om(z,x)  \]
is a function of the environments  $\{\om_{z,s}: s\le -1, \abs{z-x}\le R\abs{s}\}$. 
Consequently, if $\abs{x}>(R+1)M$,  the entire sequence $\{   f_N(T_{x,0}\om)\}_{N\in\N}$ is 
a function of the environments outside $\Lambda$, and then so is (almost surely)
the limit $f(T_{x,0}\om)$.  Since $B$ is tail measurable,   $ \mu^{\rho,\om}(B)$
is a function of $\{ f(T_{x,0}\om): \abs{x}>(R+1)M\}$ and thereby a function 
of environments outside $\Lambda$.  Since $\Lambda$ was arbitrary, 
we  conclude that $ \mu^{\rho,\om}(B)$ is  (almost surely) a tail measurable function of $\om$.  
\end{proof}

\begin{proof}[Proof  of the first part of Theorem \ref{invthm} {\rm(}except for uniqueness{\rm)}] 
Invariance of $\mu^\rho$ for the process follows by
averaging out $\om$ in the result of Lemma \ref{lm:inv1}. 
Spatial invariance,  ergodicity and tail triviality of $\mu^\rho$ are in
Lemma \ref{erglm}. That $\int \eta(0) \,d\mu^{\rho}=\rho$ follows
from the definition of $\mu^\rho$. 

We prove the ergodicity of the process $\eta_\abullet$ under the 
time-shift-invariant path measure $P_{\mu^\rho}$.  
We use the notation $\mu^\rho$ also for the joint measure
$\mu^\rho(d\om,d\eta)=\bP(d\om)\mu^{\rho,\om}(d\eta)$ and not only for 
the marginal on $\eta$.  
Let $\cJ$ be the $\sigma$-algebra of invariant sets on the state space of the
particle system:
\[ \cJ=\{ B\subseteq\bZ_+^{\bZ^d}:   \ind_B(\eta)=P_\eta\{\eta_1\in B\} 
\ \text{ for $\mu^\rho$-a.s.\ $\eta$ } \}  \]
  By Corollary 5 on p.~97 of \cite{rosenblatt} it suffices to show that $\cJ$ is trivial.  
We establish triviality of $\cJ$ by showing that 
$E^{\mu^\rho}[\psi\,\vert\,\cJ]$ is almost surely a constant for an arbitrary  bounded cylinder
function $\psi$ on $\bZ_+^{\bZ^d}$.  
 
 Let $\eta^{a,b}$ denote the configuration obtained by moving one  particle from site 
$a$ to site $b$,  if possible:  $\eta^{a,b}=\eta$ if $\eta(a)=0$, while if $\eta(a)>0$,
\[  \eta^{a,b}(x)=\begin{cases} \eta(a)-1 &x=a\\  \eta(b)+1 &x=b\\ \eta(x) &x\ne a,b. \end{cases} \]
 
\begin{lemma}  There exists a version of $E^{\mu^\rho}[\psi\,\vert\,\cJ]$
such that    for all $\eta\in\bZ_+^{\bZ^d}$  
and $a,b\in\bZ^d$,
$E^{\mu^\rho}[\psi\,\vert\,\cJ](\eta)=E^{\mu^\rho}[\psi\,\vert\,\cJ](\eta^{a,b})$.   
\label{condexplm}\end{lemma}
\begin{proof}  By Corollary 2 on p.~93 of \cite{rosenblatt}, we can define a version 
$\wt\psi$ of $E^{\mu^\rho}[\psi\,\vert\,\cJ]$ pointwise by 
\[   \wt\psi(\eta)=\varlimsup_{n\to\infty}  \frac1n\sum_{t=0}^{n-1} E_\eta[\psi(\eta_t)]. \]
We show that  $\wt\psi(\eta)=\wt\psi(\eta^{a,b})$. 

 Assume $\eta(a)>0$. Consider the basic coupling  $P_{\eta,\eta^{a,b}}$ of two processes 
$(\eta_t,\zeta_t)$ with initial configurations
$(\eta_0,\zeta_0)=(\eta,\eta^{a,b})$, as described in Section \ref{coupsec}.   
Let  \[ \sigma=\inf\{t: \text{$\psi(\eta_s)=\psi(\zeta_s)$ for all $s\ge t$} \}. \]

We observe that $P_{\eta,\eta^{a,b}}\{\sigma<\infty\}=1$ in all dimensions.  
In dimensions $d\in\{1,2\}$  the  irreducible  $\qhom$-random walk is recurrent, 
hence the two 
discrepancies of opposite sign that start at $a$ and $b$ annihilate with probability 1.  
In dimensions $d\ge 3$ the discrepancies are marginally genuinely $d$-dimensional random 
walks by assumption \eqref{Rass}.  Thus  they are transient, and so either the discrepancies
annihilate or eventually they never return to the finite set of sites that support $\psi$.

The conclusion of the lemma follows:  
  \begin{align*}
&\abs{ E_\eta[\psi(\eta_t)]-E_{\eta^{a,b}}[\psi(\eta_t)]}  
 \le  2\norm{\psi}_\infty 
  P_{\eta,\eta^{a,b}}\{\sigma>t\}\longrightarrow 0 \quad\text{as $t\to\infty$.} 
\qedhere \end{align*}
\end{proof}
 
\begin{lemma}  Suppose $h$ is a bounded measurable function on $\bZ_+^{\bZ^d}$
 such that for all $a,b\in\bZ^d$,
  $h(\eta^{a,b})=h(\eta)$   $\mu^\rho$-a.s.
 Then there exists a tail measurable function $h_1$ such that $h=h_1$ $\mu^\rho$-a.s.
\label{mupartlm}\end{lemma}
 
\begin{proof}  To show approximate tail measurability we approximate by a cylinder
function and then move  particles far enough one by one.  
  (We learned this trick from \cite{sethIHP2001}.) 
Let $\eta^a$ denote the configuration obtained by removing one  particle from site
$a$ if possible:
\[  \eta^a(x)=\begin{cases} (\eta(a)-1)^+ &x=a\\  \eta(x) &x\ne a. \end{cases} \]

Let $\e>0$.  Pick a bounded cylinder function $\tilde h$ such that 
$E^{\mu^\rho}\abs{h-\tilde h}^2<\e^2$. 
For each $\om$ pick 
  $b(\om)\in\bZ^d$ so that  $f(T_{b(\om),0}\om)\ge 1/4$ 
  and  $\tilde h$ does not depend on the coordinate $\eta(b(\om))$.  Such $b(\om)$ 
  exists a.s.\ by the ergodic theorem since $\bE f=1$.  Choose 
  $b(\om)$ so that it is a measurable function.
Since $h(\eta)=h(\eta^{a,b(\om)})$ $\mu^\rho(d\om,d\eta)$-a.s.\ and $\tilde{h}(\eta^{a})=
\tilde{h}(\eta^{a,b(\om)})$ by choice of $b(\om)$, 
\begin{align*}  \int  \abs{h(\eta)-h(\eta^{a})}\,\mu^{\rho}(d\eta)  &\le \int \ind_{\{\eta(a)>0\}} \abs{h(\eta^{a,b(\om)})
-\tilde{h}(\eta^{a,b(\om)})}\,\mu^{\rho}(d\om, d\eta)\\
&\qquad\qquad \qquad\;+\;
\int \ind_{\{\eta(a)>0\}} \abs{\tilde{h}(\eta^{a})-h(\eta^{a})}\,\mu^{\rho}(d\eta). 
\end{align*}

In the next calculation we bound the first integral after the inequality. 
Write $\eta=(\eta',\eta(a),\eta(b(\om)))$ to make the
coordinates at $a$ and $b(\om)$ explicit. Change summation indices and apply Cauchy-Schwarz: 
\begin{align*}
&\int \ind_{\{\eta(a)>0\}} \abs{h(\eta^{a,b(\om)})-\tilde{h}(\eta^{a,b(\om)})}\,\mu^{\rho}(d\om, d\eta)\\
&=\bE \sum_{\substack{k>0\\ \ell\ge 0}} \pss^{\rho f(T_{a,0}\om)}(k)  
\pss^{\rho f(T_{b(\om),0}\om)}(\ell) 
\int \abs{h(\eta',k-1,\ell+1)\\
&\qquad\qquad \qquad\qquad \qquad\qquad \qquad\qquad
-\tilde{h}(\eta',k-1,\ell+1)}\,\mu^{\rho,\om}(d\eta')\\[3pt]
&=\bE \sum_{\substack{m\ge 0\\ n>0}}  \frac{f(T_{a,0}\om)}{m+1}\cdot\frac{n}{f(T_{b(\om),0}\om)}\cdot 
\pss^{\rho f(T_{a,0}\om)}(m)  \pss^{\rho f(T_{b(\om),0}\om)}(n) \\
&\qquad\qquad \qquad\qquad
\times \int \abs{h(\eta',m,n)-\tilde{h}(\eta',m,n)}\,\mu^{\rho,\om}(d\eta')\\[3pt]
&=\int  \ind_{\{\eta(b(\om))>0\}}  \frac{f(T_{a,0}\om)}{\eta(a)+1}\cdot\frac{\eta(b(\om))}{f(T_{b(\om),0}\om)}\cdot 
\abs{h(\eta)-\tilde{h}(\eta)}\,\mu^{\rho}(d\om, d\eta)\\
&\le \biggl\{ \bE \sum_{\substack{m\ge 0\\ n>0}}  \frac{f(T_{a,0}\om)^2}{(m+1)^2}
\cdot\frac{n^2}{f(T_{b(\om),0}\om)^2}\cdot 
\pss^{\rho f(T_{a,0}\om)}(m)  \pss^{\rho f(T_{b(\om),0}\om)}(n) \biggr\}^{1/2}     \\
&\qquad\qquad \qquad\qquad
\times\bigl\{ E^{\mu^\rho}\abs{h-\tilde h}^2\bigr\}^{1/2} \\
  &\le \sqrt{5\E[f^2]}\e.
\end{align*}
To obtain the second equality we replace $\Gamma^{\rho f(T_{a,0}\om)}(k)$ by  $\Gamma^{\rho f(T_{a,0}\om)}(k-1)\cdot  \frac{f(T_{a,0}\om)}{k}$ and similarly for $\Gamma^{\rho f(T_{b(\om),0}\om)}(l)$.

An analogous argument (but easier since we do not need the $b(\om)$)  gives 
\[  \int \ind_{\{\eta(a)>0\}} \abs{\tilde{h}(\eta^{a})-h(\eta^{a})}\,\mu^{\rho}(d\eta)
\le C\e.  \]
Since $\e>0$ was arbitrary we have $h(\eta)=h(\eta^{a})$ $\mu^\rho$-a.s. 
 
 For  given
 finite $\Lambda\subseteq\bZ^d$,  applying the mapping $\eta\mapsto\eta^{a}$ repeatedly  to remove all particles from $\Lambda$ shows that  
$h$ equals a.s.\ a function $g_\Lambda$  that does not depend on   $(\eta(x):x\in\Lambda)$. 
As $\Lambda\nearrow\Z^d$   along cubes, the limit $h_1=\lim g_\Lambda$ exists a.s.\ by martingale convergence and  is tail measurable. 
  \end{proof}

We can now conclude the proof of (temporal) ergodicity of the process $\eta_\abullet$.   Lemmas \ref{condexplm} and \ref{mupartlm}
show that $E^{\mu^\rho}[\psi\,\vert\,\cJ]$ is $\mu^\rho$-a.s.\ tail measurable, and  
hence a constant by Lemma \ref{erglm}.
\end{proof} 

\subsection{Proof of uniqueness}
In this subsection, we complete the proof of part (a) of Theorem \ref{invthm} by showing that $\mu^{\rho}$ is the \textit{unique} invariant distribution with the stated properties. We also prove the second part of Theorem \ref{invthm}.  The proof of uniqueness uses standard techniques of  interacting particle systems \cite{ligg-85}.  We will arrive at the proof of uniqueness through a sequence of lemmas.

For two configurations $\eta, \zeta$ of occupation variables, we say that $\eta \le \zeta $ if $\eta(x)\le \zeta(x)$ for all $x$. For two probability distributions $\mu, \nu$ on the configuration space, we say $\mu \le \nu $ if there exists a probability measure $\wt{\mu}$ on pairs $(\eta, \zeta)$ of configurations of occupation variables such that
$\wt\mu(\eta \le \zeta)=1$
and the marginals of $\wt\mu$  are $\mu$ and $\nu$. For a convex set $\mathcal{A}$, $\mathcal{A}_e$ will denote the set of extremal elements.

Recall that $\wt{\cS}$, resp.\ $\wt{\cS}_e$, 
denotes the set of spatially invariant resp.\ ergodic probability distributions on 
   pairs $(\eta,\zeta)$ of configurations of occupation variables. 
   Let $\tilde{\mathcal{I}}$ denote the set of probability distributions on pairs of configurations of occupation variables, that are invariant under the  temporal evolution 
   described at the beginning of Section \ref{coupsec}.     

\begin{lemma}
\label{l1}
If $\rho_1<\rho_2$ then $\mu^{\rho_1} \le \mu^{\rho_2}$.
\end{lemma}
\begin{proof}  We couple $\mu^{\rho_1,\w}$ and $\mu^{\rho_2,\w}$ by letting $\wt\mu^\w$ be the distribution of $(\eta,\zeta)$ defined by letting occupation variables  $\eta(x)$ be  independent Poisson with means $\rho_1f(T_{x,0}\om)$,  $\gamma(x)$ be independent Poisson with means $(\rho_2-\rho_1)f(T_{x,0}\om)$, and then setting  $\zeta(x)=\eta(x)+\gamma(x)$.   Then  define the coupling of $\mu^{\rho_1}$ and $\mu^{\rho_2}$ by $\wt\mu(\cdot)=\E[\wt\mu^\w(\cdot)]$.  
\end{proof}

We state the next two lemmas without proof. The proofs can be found in Lemmas 4.2 - 4.5 of \cite{andj-82}.
\begin{lemma}
\label{l2}We have 

{\rm (a)} If $\mu_1,\mu_2 \in \cI \cap \cS$, there is a $\wt{\mu}\in \wt{\cI}\cap \wt{\cS}$  with marginals $\mu_1$ and $\mu_2$.

{\rm (b)} If $\mu_1,\mu_2 \in (\cI\cap \cS)_e$, there is a $\wt{\mu}\in (\wt{\cI}\cap \wt{\cS})_e$  with marginals $\mu_1$ and $\mu_2$.
\end{lemma}

\begin{lemma}
\label{l2.5}
If $ \wt{\mu} \in (\wt{\cI} \cap \wt{\cS})_e \mbox{ and } \wt{\mu} \lbrace (\eta,\zeta): \eta \ge \zeta \mbox{ or } \zeta \ge \eta \rbrace =1$ then
\[ \wt{\mu} \lbrace (\eta,\zeta): \eta\ge \zeta\rbrace =1 \mbox{ or } \wt{\mu} \lbrace (\eta,\zeta): \zeta\ge \eta\rbrace =1.\]
\end{lemma}


A crucial lemma needed in the proof of uniqueness is the following.
\begin{lemma}
\label{l4}
 Let $\wt{\mu}\in \wt{\cS}_e \mbox{ such that } \int [\eta(0)+\zeta(0)] d\wt{\mu} < \infty$. Fix $x \ne y  \in \bZ^d$. Then 
\[ \lim_{t \to \infty} \wt{\mu}_t \big\lbrace (\eta,\zeta): \eta(x) >\zeta(x) \mbox{ and } \eta(y)< \zeta(y)\big \rbrace=0\]
\end{lemma}
\begin{proof}
Our proof employs some of the notation developed in Section \ref{coupsec}. Fix a positive integer $m$. Let $I=[-m,m]^d$ and let $B$ be the event that $I$ contains both $(+)$ and $(-)$ particles. The theorem will be proved if we can show that $\wt{\mu}_t(B) \to 0$ as $t \to \infty$. So let us assume to the contrary that we can find a sequence $t_k \uparrow \infty $ such that 
\begin{equation}
\label{opp}
 \wt{\mu}_{t_k}(B)\ge \delta >0.
\end{equation}

By our assumptions on the environment, we can find a positive integer $T=T(m)$ and a positive real number $\rho=\rho(m)>0$ such that 
\[\min_{x,\,y\, \in\, I} P\big\{ X^{x}_{\abullet} \mbox{ and }\tilde{X}^{y}_{\abullet} \mbox{ meet by time } T \bigr\}  \ge \rho. \]  
Let  $A(t,y)$ denote the event that a $(+)$ or a $(-)$ particle present in the cube $y+I$ at time $t$ has been annihilated by time $t+T$.
It is clear that 
\begin{equation} \label{lb} P\big[A(t,y)\big{\vert} \eta_t,\zeta_t  \big] \ge \rho \cdot \one_{B}\big{\{}\theta_y(\eta_t,\zeta_t)\big{\}} \mbox{ a.s. }\end{equation}
For what follows, assume that all $t_{k+1}-t_k \ge T$. Let $\phi_t(x)=\beta_t^{+}(x) +\beta_t^{-}(x)$ be the number of discrepancy particles at $x$ at time $t$. Let $n=l(2m+1)+m$ for a positive integer $l$ and divide the cube $[-n,n]^d$ into 
$(2l+1)^d$ cubes of side length $2m+1$. We have 
\begin{align*} \frac{1}{(2n+1)^d} \sum_{y\in [-n,n]^d} \phi_{t_k+T}(y) &\le \frac{1}{(2n+1)^d} \sum_{y\in [-n-RT,n+RT]^d} \phi_{t_k}(y)\\
& \qquad-\, \frac{1}{(2n+1)^d} \sum_{j=1}^{(2l+1)^d}\one_B\{\theta_{u(j)}(\eta_{t_k},\zeta_{t_k})\}\cdot \one_{A(t_k,u(j))}
\end{align*}
where $u(j)$ is the center of cube $j$. Taking expectations and letting $n \to \infty$, we get 
\[ E_{\wt{\mu}}[\phi_{t_k+T}(0)] \le  E_{\wt{\mu}}[\phi_{t_k}(0)]  -\liminf_{n \to \infty} \frac{-1}{(2n+1)^d} \sum_{j=1}^{(2l+1)^d} E_{\wt{\mu}}\big[ \one_B\{\theta_{u(j)}(\eta_{t_k},\zeta_{t_k})\}\cdot \one_{A(t_k,u(j))}\big] \]
It follows from \eqref{lb} and \eqref{opp} that 
\[ E_{\wt{\mu}}\big[\one_B\{\theta_{u(j)}(\eta_{t_k},\zeta_{t_k})\}\cdot \one_{A(t_k,u(j))}\big] \ge \rho \wt{\mu}_{t_k} (B) \ge \rho \delta.\]
We thus have 
\[E_{\wt{\mu}}[\phi_{t_{k+1}}(0)]\le E_{\wt{\mu}}[\phi_{t_k+T}(0)] \le  E_{\wt{\mu}}[\phi_{t_k}(0)]  - \frac{\rho \delta}{(2m+1)^d}. \]
We can conclude from Lemma \ref{betadecrease} that $E_{\wt\mu}[\phi_{t_k}(0)]\to -\infty$. But this is a contradiction since $E_{\wt{\mu}}[\phi_t(0)]\ge 0$. 
The proof of the lemma is complete.
\end{proof}

\begin{lemma}
\label{l5}
If $\mu_1,\mu_2 \in (\cI\cap \cS)_e$ and $E_{\mu_i} \eta(0)  < \infty$ for $i=1,2$, then $\mu_1 \le \mu_2$ or $\mu_2 \le \mu_1$.
\end{lemma}
\begin{proof}
From Lemma \ref{l2}, we can find $\wt{\mu} \in (\wt{\cI}\cap \wt{\cS})_e$ with marginals $\mu_1$ and $\mu_2$. Using the ergodic decomposition of stationary measures \cite[Theorem 6.6]{vara-book1},
\begin{align*} 
& \wt{\mu}\big\{(\eta,\zeta):\eta(x)>\zeta(x) \mbox{ and } \eta(y)<\zeta(y)\big\}\\
&\qquad 
= \int_{\wt\cS_e} \tilde{\nu}\big\{(\eta,\zeta):\eta(x)>\zeta(x) \mbox{ and } \eta(y)<\zeta(y)\big\} \Psi(d\tilde\nu), \end{align*} 
for a probability measure $\Psi$ on $\wt\cS_e$. On applying the operator $S(t)$ to both sides of the above equation, we observe that the right hand side goes to $0$. We thus get 
\[\wt{\mu}\lbrace (\eta,\zeta): \eta\le \zeta \mbox{ or } \zeta\le \eta \rbrace =1\]
An application of Lemma \ref{l2.5} completes the proof.
\end{proof}

\begin{proposition}\label{mu=murho}
If $ \mu \in (\cI\cap \cS)_e$ and $\rho_0=E_{\mu} \eta(0) < \infty$ then $\mu=\mu^{\rho_0}$.
\end{proposition}
\begin{proof}
Since $\mu^{\rho} \in \cS_e\cap \cI,$ it follows that $\mu^{\rho} \in (\cI\cap \cS)_e.$ We can then conclude from Lemma \ref{l1} and Lemma \ref{l5} that there exists a $\rho_0'\in[0,\infty]$ such that $\mu \le \mu^{\rho}$ for $\rho > \rho_0'$ and $\mu \ge  \mu^{\rho}$ for $\rho < \rho_0'$. In particular, we have $\rho_0=E_\mu\eta(0)\le\rho$ for $\rho>\rho_0'$ and similarly $\rho_0\ge\rho$ for $\rho<\rho_0'$. This says that $\rho_0'=\rho_0$.

Now fix $\rho_1<\rho_0<\rho_2$. For all $(x_1,x_2,\cdots,x_n) \in (\bZ^d)^n$ and all $(k_1,k_2,\cdots,k_n)\in (\bZ_{+})^n$, we have
\[ \mu^{\rho_1} ( \eta(x_i)\ge k_i, 1\le i \le n )\le  \mu ( \eta(x_i)\ge k_i, 1\le i \le n )\le  \mu^{\rho_2} ( \eta(x_i)\ge k_i, 1\le i \le n ) \]
The first inequality (resp. the second inequality) above can be seen by looking at the coupled measure $\wt{\mu}$ corresponding to $\mu^{\rho_1}$ (resp. $\mu^{\rho_2}$) and $\mu$ so that $\wt{\mu}(\eta\le \zeta)=1$ (resp. $\wt{\mu}(\eta\ge \zeta)=1$). Now let $\rho_1 \uparrow \rho_0$ and $\rho_2\downarrow \rho_0$ to see that $\mu $ has the same finite dimensional distributions as $\mu^{\rho_0}$.
\end{proof}


\begin{proof}[Proof  of the remaining parts of Theorem \ref{invthm}]  We first prove that $\mu^{\rho}$ is the unique measure with the stated properties in part (a) of Theorem \ref{invthm}. Indeed, let $\mu$ be another measure with those properties. Since $\mu \in \mathcal{S}_e \cap \mathcal{I}$, we can conclude that $\mu \in (\mathcal{I} \cap \mathcal{S})_e$. From Proposition \ref{mu=murho}, we must have that $\mu=\mu^{\rho}$.

We now turn to part (b) of the theorem. Let $\nu$ be a probability measure on $\bZ_+^{\bZ^d}$ that 
is stationary and ergodic under spatial translations 
and has mean occupation $\int \zeta(0)\,d\nu=\rho$.  
Denote   the
occupation  process with initial distribution $\nu$  by $\zeta_t$.   
Utilizing the ergodic decomposition theorem 
\cite[Theorem 6.6]{vara-book1}, find 
$\wt\mu\in\wt\cS_e$ with marginals  $\mu^\rho$ and $\nu$. 
Let $\wt\mu_t$ be the time $t$ distribution of the
joint process $(\eta_t, \zeta_t)$ coupled as described  in 
Section \ref{coupsec}. 

Initial shift invariance implies that mean occupations
are constant $\rho$ throughout time and space:
\begin{align*}
E_{\wt{\mu}_t}[\zeta(x)]=E_\nu[ \zeta_t(x)] &= \int \bE\Bigl\{ \sum_{y} \zeta(y) 
\pi_{0,t}^{\omega}(y,x) \Bigr\} \, \nu(d\zeta) 
= \sum_{y} \mathbb{E}\big(\pi_{0,t}^{\omega}(y,x)\big) 
\int \zeta(y) d \nu =\rho.
\end{align*}
Chebyshev's inequality and Tychonov's theorem (Thm.~37.3 in \cite{munk}) can be used
to show that the sequence $\{\wt\mu_t\}_{t\in\bZ_+}$ is
tight.  

  Let $\tilde\nu$ be any limit point
as $t\to\infty$.  Then by Proposition \ref{rho1gerho2}
$ \tilde{\nu}{\lbrace}(\eta,\zeta): \eta=\zeta {\rbrace} =1.$ 
This proves that $P_\nu\{\zeta_t\in\cdot\}\Rightarrow\mu^\rho$.
This completes the proof of Theorem \ref{invthm}.
\end{proof}


\section{Covariances of the invariant measures}  \label{covarsec} 

 Define the Green's functions for both $q$ and $\qhom$ walks by 
\[
G_N(x,y)=\sum_{k=0}^{N} q^k(x,y)
\quad\text{and}\quad   \Ghom_N(x,y)=\sum_{k=0}^{N} \qhom^k(x,y).  \]
Recall
the potential kernel for the $\qhom$ walk
 \be 
 \ahom(x)=\lim_{N\to\infty}\bigl( \Ghom_N(0,0)-\Ghom_N(x,0)\bigr). \label{abarY}\ee
 In the transient case $d\ge 3$  the limit above exists trivially, 
 since  
 \[ 
G(x,y)=\sum_{k=0}^{\infty} q^k(x,y)<\infty 
\quad\text{and}\quad   \Ghom(x,y)=\sum_{k=0}^{\infty} \qhom^k(x,y)<\infty.  \] 
So for $d\ge 3$ 
\be   \ahom(x)= \Ghom(0,0)-\Ghom(x,0).  \label{abarYtrans}\ee
  For the existence of the limit \eqref{abarY}   in the recurrent case  $d\in\{1,2\}$
  see T1 on p.~352 of \cite{spitzer}.  
In all cases the kernel $\ahom(x)$  satisfies these equations:
 \begin{align}\label{a:eqn}  
 \sum_z \qhom(0,z)\ahom(z)=1 \quad\text{and}\quad  
  \sum_z \qhom(x,z)\ahom(z)=\ahom(x) \ \text{for $x\ne 0$.}
  \end{align}

 The  constant $\beta$ defined   by \eqref{defbeta}   has the alternate 
 representation 
\be \beta=\sum_z q(0,z)\ahom(z).  \label{beta2}\ee
We omit the argument for the equality of the two representations
of $\beta$.  It    is a simple version of the one given at the end of this section for
\eqref{covfm2}.

To prove Theorem \ref{covthm}  we first verify this proposition and then 
derive the Fourier representation \eqref{covfm2}.  

\begin{proposition}  Let $d\ge 1$. 
For $m\in\bZ^d\setminus\{0\}$ 
\begin{align}
  \Cvv[f(\om), f(T_{m,0}\om)] &= \beta^{-1} \sum_z q(0,z)[\ahom(-m)-\ahom(z-m)] 
 \label{covfm} 
\end{align}
 and 
 \be  \Vvv[f(\om)]=\beta^{-1}-1. \label{varf}\ee
\label{covprop}\end{proposition}

 A few more notations.   Recall that $Y_n$ denotes the Markov chain with transition $q$
 and $\Yhom_n$ the $\qhom$ random walk.  
   Successive returns to the origin are 
marked as follows:  
\be \text{$\tau_0=0$ and  for $j>0$,   
 $\tau_j=\inf\{n>\tau_{j-1}: Y_n=0\}$.} \label{tau1}\ee
 Abbreviate $\tau=\tau_1$.   The corresponding stopping time
 for $\Yhom_n$ is $\tauhom$.  
For $m\in\bZ^d$ and $N\ge 1$ abbreviate
\[  C_N(m)=  \Cvv[f_N(\om), f_N(T_{m,0}\om)]
=\sum_{z,w\in\bZ^d}\Cvv\bigl[   \pi_{-N,0}(z,0),\,
  \pi_{-N,0}(w,m) \bigr]. \]
Define also the function 
\[
\hfun(y)=\sum_{z\in\bZ^d} \Cvv[\pi_{0,1}(0,y+z),\,\pi_{0,1}(0,z)]=q(0,y)-\qhom(0,y), \quad y\in\bZ^d. \]
Symmetry $h(-y)=h(y)$ holds. 

\begin{lemma}  In all dimensions $d\ge 1$, 
\be C_N(m)=\sum_{y\in\bZ^d} \hfun(y) G_{N-1}(y,m).  \label{covar1a}\ee
\end{lemma}

\begin{proof} The case $N=1$ follows from a shift of space and time. 
To do induction on $N$   use the Markov property and the 
additivity of covariance.  Abbreviate temporarily $\kappa_{x,y}= \pi_{-N,-N+1}(x,y)$ and 
recall that the mean kernel is $p_{y-x}=\bE\kappa_{x,y}$.
\begin{align}
C_N&(m)=\sum_{z,z_1,w,w_1}\Cvv\bigl[ \kappa_{z,z_1} \pi_{-N+1,0}(z_1,0)\,,\,
\kappa_{w,w_1} \pi_{-N+1,0}(w_1,m) \bigr]\nn\\
&=\sum_{z,z_1,w,w_1} \Bigl\{    \Cvv\bigl[ (\kappa_{z,z_1}-p_{z_1-z}) \pi_{-N+1,0}(z_1,0)\,,\,
(\kappa_{w,w_1}-p_{w_1-w}) \pi_{-N+1,0}(w_1,m) \bigr] \label{covline1}\\[3pt]
&\quad\qquad + \Cvv\bigl[ p_{z_1-z}   \pi_{-N+1,0}(z_1,0)\,,\,
(\kappa_{w,w_1}-p_{w_1-w}) \pi_{-N+1,0}(w_1,m) \bigr] \label{covline2}\\[4pt]
&\qquad \qquad\quad + \Cvv\bigl[  p_{z_1-z} \pi_{-N+1,0}(z_1,0)\,,\,
 p_{w_1-w} \pi_{-N+1,0}(w_1,m) \bigr]\,\Bigr\}. \label{covline3}
\end{align}
Working from the bottom up, the terms on line \eqref{covline3} add up to 
$C_{N-1}(m)$.  The terms on  line \eqref{covline2} vanish because $\kappa_{w,w_1}-p_{w_1-w}$
is mean zero and independent of the other random variables inside the covariance. 
On line \eqref{covline1} the covariance vanishes unless $z=w$. Thus by 
rearranging line \eqref{covline1} we get
\begin{align*}
&C_N(m)-C_{N-1}(m)= {\rm line \,\eqref{covline1}} \\ 
&=\sum_{z,z_1, w_1}   \Cvv  (\kappa_{z,z_1}, \kappa_{z,w_1} ) 
\bE\bigl[ \pi_{-N+1,0}(z_1,0)  \pi_{-N+1,0}(w_1,m) \bigr]  \\
&=\sum_{y,x}  \Cvv  (\kappa_{0,x}, \kappa_{0,x+y} ) 
\sum_\ell \bE\bigl[ \pi_{-N+1,0}(y,m+\ell)  \pi_{-N+1,0}(0,\ell) \bigr]\\
 &=\sum_y \hfun(y) q^{N-1}(y,m).  
\qedhere\end{align*}
\end{proof}

In the recurrent case we will use Abel summation, hence the next lemma.  

 \begin{lemma} Let $d\in\{1,2\}$. 
For $x,m\in\bZ^d$,  the limit 
\be a(x,m)=\lim_{s\nearrow 1}\sum_{k=0}^\infty s^k\bigl( q^k(0,m)- q^k(x,m)\bigr)  \label{axmAbel}\ee
exists. For $m=0$ the limit is 
\be a(x,0)=\frac{\ahom(x)}{\beta} \label{axm0}\ee
and  for $m\ne 0$ 
\be a(x,m)=\frac{\ahom(x)}{\beta}\sum_{z}q(0,z)\bigl[\ahom(-m)-\ahom(z-m)\bigr]
-\ahom(-m)+\ahom(x-m).  \label{axm}\ee
\end{lemma}
 
 \begin{proof} Let $s$ vary in $(0,1)$ and let  
\[
U(x,m,s)=E_x\Bigl[ \,\sum_{k=0}^{\tau-1} s^k\ind\{Y_k=m\}\Bigr]
\underset{s\nearrow 1}{\longrightarrow}
E_x\Bigl[ \,\sum_{k=0}^{\tau-1}  \ind\{Y_k=m\}\Bigr]=U(x,m). 
 \]
Decompose the  summation across intervals $[\tau_j,\tau_{j+1})$ and use the Markov 
property: 
\begin{align*}
\sum_{k=0}^\infty s^k q^k(x,m)&= 
E_x\Bigl[ \,\sum_{k=0}^{\tau_1-1} s^k\ind\{Y_k=m\}\Bigr]
+ \sum_{j=1}^\infty E_x\Bigl[ s^{\tau_j}\sum_{k=\tau_j}^{\tau_{j+1}-1} s^{k-\tau_j}\ind\{Y_k=m\}\Bigr]\\
&= U(x,m,s)+ \sum_{j=1}^\infty E_x(s^{\tau}) E_0(s^{\tau})^{j-1} U(0,m,s)\\
&=U(x,m,s)+ \frac{ E_x(s^{\tau})  }{ 1-E_0(s^{\tau})}U(0,m,s). 
\end{align*} 
From this,
\be
\sum_{k=0}^\infty s^k\bigl( q^k(0,m)- q^k(x,m)\bigr)
=\frac{ 1-E_x(s^{\tau})  }{ 1-E_0(s^{\tau})}U(0,m,s) -U(x,m,s).
\label{covaraux7}\ee 

We analyze the quantities on the right in \eqref{covaraux7}.  

Suppose first  $x\ne 0$.  
Then $U(x,m)$ is the same for the Markov chain $Y_k$ as for the random walk 
$\Yhom_k$ because these processes agree until the first visit to $0$. 
In the notation of Spitzer \cite{spitzer}, with a check added to refer to the random walk
$\Yhom_k$,   $\ghom_{\{0\}}(x,m)=U(x,m)$.  By {P29.4} on p.~355 of \cite{spitzer}
and {D11.1} on p.~115, 
for recurrent random walk 
\[
U(x,m)=\ghom_{\{0\}}(x,m)
=\ahom(x)+\ahom(-m)-\ahom(x-m). \] 

For $x=0$ we have $U(0,0)=1$, and for $m\ne 0$,
\begin{align*}
U(0,m)&=\sum_{y\ne 0} q(0,y) U(y,m)
=\sum_{y\ne 0} q(0,y)\bigl[ \ahom(y)+\ahom(-m)-\ahom(y-m) \bigr]\\
&=\beta + \sum_{y} q(0,y)\bigl[\ahom(-m)-\ahom(y-m) \bigr]. 
\end{align*}

For the asymptotics of the fraction on the right in \eqref{covaraux7} 
we can assume again $x\ne 0$ for otherwise the value is 1. It will be convenient to look at the reciprocal.  
A computation gives 
\begin{align*}
&\frac{ 1-E_0(s^{\tau})  }{ 1-E_x(s^{\tau})} = 
\frac{\sum_{k=0}^\infty s^k P_0(\tau>k)}{\sum_{k=0}^\infty s^k P_x(\tau>k)}\\
&\quad = \frac1{\sum_{k=0}^\infty s^k P_x(\tau>k)} \;+\; 
s\sum_{z\ne 0}q(0,z)\frac{\sum_{k=0}^\infty s^k P_z(\tau>k)}{\sum_{k=0}^\infty s^k P_x(\tau>k)}. 
\end{align*}
Again we can take advantage of known random walk limits because both $x,z\ne 0$ so
the probabilities are the same as those for $\Yhom_k$.  By {P32.2} on p.~379 of \cite{spitzer},
as $s\nearrow 1$, for recurrent random walk the above converges to (note that $E_x(\tau)=\infty$)
\[ \sum_{z\ne 0}q(0,z)\frac{\ahom(z)}{\ahom(x)}=\frac\beta{\ahom(x)}. \]

Letting $s\nearrow 1$ in \eqref{covaraux7} gives   \eqref{axm0} and  \eqref{axm}.
 \end{proof}
 
For $m=0$ we can obtain the convergence as in \eqref{abarY} without the Abel summation.  
But we do not need this   for further development. 

\begin{proof}[Proof of Proposition \ref{covprop}]
Since $f_N\to f$ in $L^2(\bP)$, the covariance in \eqref{covfm} is given by the 
 limit of $C_N(m)$, so by   \eqref{covar1a}
\begin{align*}
 &\Cvv[f(\om), f(T_{m,0}\om)] = 
 \lim_{N\to\infty} \Bigl\{  \sum_y  q(0,y) G_{N-1}(y,m) - \sum_y  \qhom(0,y) G_{N-1}(y,m)\Bigr\}.
\end{align*}
Next, 
\begin{align*} 
 \sum_y  q(0,y) G_{N-1}(y,m)=G_N(0,m)-\delta_{0,m} 
 =q^N(0,m)-\delta_{0,m}+G_{N-1}(0,m).
\end{align*}
Since the  Markov chain $q$ follows the random walk $\qhom$ away from $0$
it is  null recurrent for $d=1,2$  and transient for $d\ge 3$. 
So $q^N(0,m)\to 0$ 
\cite[Theorem 1.8.5]{norr}.  Thus the limiting covariance now has the form
\be   -\delta_{0,m} +  \lim_{N\to\infty}  \sum_y  \qhom(0,y) [G_{N-1}(0,m)-G_{N-1}(y,m)]. 
\label{C3}\ee
At this point the treatment separates into recurrent and transient cases. This is because the Green's function is uniformly bounded only in the transient case.

\medskip

{\bf Case 1.}  $d\in\{1,2\}$

\medskip

Convergence in \eqref{C3}  implies Abel convergence (Theorem 12.41 in \cite{whee-zygm}
or Theorem 1.33 in Chapter III of \cite{zygm}), so the limiting covariance equals 
\[
 -\delta_{0,m} + \lim_{s\nearrow 1}    \sum_y  \qhom(0,y)  \sum_{k=0}^\infty s^k\bigl( q^k(0,m)- q^k(y,m)\bigr). \]
 By substituting in \eqref{axm0} and \eqref{axm} we obtain \eqref{varf} and \eqref{covfm}.

\medskip

{\bf Case 2.}  $d\ge 3$

\medskip

In the transient case we can pass directly to the limit in \eqref{C3} and obtain  
 \be
 \Cvv[f(\om), f(T_{m,0}\om)]   = -\delta_{0,m} +   \sum_y  \qhom(0,y) [G(0,m)-G(y,m)]. 
\label{C4}\ee
The sum above can be restricted to $y\ne 0$. By restarting after the first return 
 to $0$,   
\be    G(y,m)=  E_y\Bigl[\;\sum_{k=0}^{\tau-1} \ind\{Y_k=m\}\Bigr]
+P_y(\tau<\infty) G(0,m).  \label{C5}\ee
Next,
\be \begin{aligned}
G(0,m)&= \sum_{j=0}^\infty E_0\Bigl[  \ind\{\tau_j<\infty\}
\sum_{k=\tau_j}^{\tau_{j+1}-1}  \ind\{Y_k=m\}\Bigr]\\
&= \sum_{j=0}^\infty P_0  (\tau<\infty)^j   E_0\Bigl[ \,\sum_{k=0}^{\tau -1}  \ind\{Y_k=m\}\Bigr]\\
&=\frac1{P_0  (\tau=\infty)}\Bigl(  \delta_{0,m}   + (1-\delta_{0,m}) \sum_{z\ne 0} q(0,z)
E_z\Bigl[ \,\sum_{k=0}^{\tau -1}  \ind\{Y_k=m\}\Bigr] \, \Bigr). 
\end{aligned}\label{C6}\ee
Now consider first $m\ne 0$.  Combining the above,
\begin{align*}
 &\Cvv[f(\om), f(T_{m,0}\om)]   =    \sum_{y\ne 0}  \qhom(0,y) \bigl\{ G(0,m)-G(y,m)\bigr\}\\
 &\quad= \sum_{y\ne 0}  \qhom(0,y) \Bigl\{  \; \frac{P_y(\tau=\infty)}{P_0(\tau=\infty)}
 \sum_{z\ne 0} q(0,z)  E_z\Bigl[ \,\sum_{k=0}^{\tau -1}  \ind\{Y_k=m\}\Bigr]
- E_y\Bigl[\;\sum_{k=0}^{\tau-1} \ind\{Y_k=m\}\Bigr]\; \Bigr\} 
\intertext{using equality of $q$ and $\qhom$ away from $0$}
 &\quad= \sum_{y\ne 0}  \qhom(0,y) \Bigl\{  \; \frac{P_y(\bar\tau=\infty)}{P_0(\tau=\infty)}
 \sum_{z\ne 0} q(0,z)  E_z\Bigl[ \,\sum_{k=0}^{\bar\tau -1}  \ind\{\Yhom_k=m\}\Bigr]
- E_y\Bigl[\;\sum_{k=0}^{\bar\tau-1} \ind\{\Yhom_k=m\}\Bigr]\; \Bigr\} \\
 &\quad=   \frac{P_0(\bar\tau=\infty)}{P_0(\tau=\infty)}
 \sum_{z\ne 0} q(0,z)  E_z\Bigl[ \,\sum_{k=0}^{\bar\tau -1}  \ind\{\Yhom_k=m\}\Bigr]
- \sum_{y\ne 0}  \qhom(0,y) E_y\Bigl[\;\sum_{k=0}^{\bar\tau-1} \ind\{\Yhom_k=m\}\Bigr] 
\intertext{applying \eqref{C5} and \eqref{C6} to the $\qhom$ walk}
 &\quad=  \frac{P_0(\bar\tau=\infty)}{P_0(\tau=\infty)}
 \sum_{z\ne 0} q(0,z) \Bigl\{ \Ghom(z,m) - P_z(\bar\tau<\infty) \Ghom(0,m) \Bigr\} 
-  P_0(\bar\tau=\infty) \Ghom(0,m)\\
 &\quad=  \frac{P_0(\bar\tau=\infty)}{P_0(\tau=\infty)}
 \sum_{z\ne 0} q(0,z) \Bigl\{ \Ghom(z,m) - P_z(\bar\tau<\infty) \Ghom(0,m) \Bigr\} \\
&\qquad\qquad \qquad\qquad\qquad\qquad -     \frac{P_0(\bar\tau=\infty)}{P_0(\tau=\infty)}  
 \sum_{z\ne 0} q(0,z)  P_z(\tau=\infty)   \Ghom(0,m)\\
 &\quad=  \frac{P_0(\bar\tau=\infty)}{P_0(\tau=\infty)}
 \sum_{z\ne 0} q(0,z) \bigl[ \Ghom(z,m) -   \Ghom(0,m)  \bigr].  
\end{align*}
To finish this case, note that
\begin{align*}
 \beta&=\sum_z q(0,z)\ahom(z) =\sum_{z\ne 0} q(0,z) (\Ghom(0,0)-\Ghom(z,0)) 
 =\sum_{z\ne 0} q(0,z)  \frac{P_z(\bar\tau=\infty)}{P_0(\bar\tau=\infty)}\\
 &= \frac{P_0(\tau=\infty)}{P_0(\bar\tau=\infty)}.
\end{align*}
We have arrived at 
\[ \Cvv[f(\om), f(T_{m,0}\om)]  = \beta^{-1}
 \sum_{z\ne 0} q(0,z) \bigl[  \ahom(-m)  -\ahom (z-m)\bigr].  \]   
 
Return to \eqref{C4}--\eqref{C6} to cover the case $m=0$:
\begin{align*}
&\Cvv[f(\om), f(\om)]   =    \sum_y  \qhom(0,y) [G(0,0)-G(y,0)] -1
=  \sum_{y\ne 0}  \qhom(0,y)  \frac{P_y(\tau=\infty)}{P_0(\tau=\infty)} -1\\
&\quad =  \frac{P_0(\bar\tau=\infty)}{P_0(\tau=\infty)} -1=\beta^{-1}-1.  
\end{align*} 

This completes the proof of Proposition \ref{covprop}.   
\end{proof}

 \begin{proof}[Completion of the proof of Theorem \ref{covthm}] 
It remains to prove the Fourier representation \eqref{covfm2} from  \eqref{covfm}.  
In several stages symmetry of $\ahom$ and the transitions  is used.
\begin{align*}
 &\Cvv[f(\om), f(T_{m,0}\om)]   = \beta^{-1} \sum_z q(0,z)[\ahom(m)-\ahom(m-z)] \\
&\quad= \lim_{N\to\infty} \beta^{-1} \sum_{k=0}^N \sum_z q(0,z)[\qhom^k(m-z,0)-\qhom^k(m,0)] \\
&\quad= \lim_{N\to\infty} \frac{\beta^{-1}}{(2\pi)^d} \sum_{k=0}^N \sum_z q(0,z)
\int_{\bT^d}[e^{-i\theta\cdot(m-z)}- e^{-i\theta\cdot m}  ]\chfnhom^k(\theta)\,d\theta \\
&\quad= \lim_{N\to\infty} \frac{-\beta^{-1}}{(2\pi)^d}  \int_{\bT^d}   \cos(\theta\cdot m)
 \frac{1-\chfn(\theta)}{1-\chfnhom(\theta)}  (1-\chfnhom^{N+1}(\theta))\,d\theta \\
&\quad=  \frac{-\beta^{-1}}{(2\pi)^d}  \int_{\bT^d}   \cos(\theta\cdot m)
 \frac{1-\chfn(\theta)}{1-\chfnhom(\theta)}  \,d\theta. 
 \end{align*}
 The last equality comes from  $0\le \chfnhom(\theta)<1$ 
 for $\theta\in\bT^d\setminus\{0\}$ and  dominated convergence. The ratio 
 $(1-\chfn(\theta))/{(1-\chfnhom(\theta))} $   stays bounded as $\theta\to 0$ because  both transitions $q$ and $\qhom$ have zero mean and $\qhom$ has a nonsingular covariance matrix    \cite[{P7} p.~74]{spitzer}.  
\end{proof}




\section{Convergence of centered current fluctuations}\label{sfdd}

We prove Theorem \ref{fddthm} by proving the following proposition. 
Recall the definition of the current $Y_n(t,r)$  from \eqref{Y}, and let 
$\{Z(t,r): (t,r)\in\bR_+\times\bR\}$
be the mean zero Gaussian process  defined by \eqref{Zint} or equivalently through the
covariance \eqref{Zcov}.  
Recall also the definitions 
 \[ \overline{Y}_n(t,r)=n^{-1/4}\bigl\{ Y_n(t,r)-E^\om[Y_n(t,r)]\bigr\},  \]
 \[ \overline Y_n(\thvec)= \sum_{i=1}^N \theta_i \overline Y_n(t_i,r_i)
\quad\text{and}\quad  Z(\thvec)= \sum_{i=1}^N \theta_i  Z(t_i,r_i). \]

\begin{proposition}
 \be   E^{\om} \bigl[ \exp \bigl\{ i \overline Y_n(\thvec) \bigr\}\bigr]
 \to \mE \bigl[\exp\bigl\{i Z(\thvec)\bigr\} \bigr]
\mbox{ in $\bP$-probability.}\label{fddlim}\ee
\label{fddprop}\end{proposition}

The remainder of the section proves this proposition and thereby Theorem \ref{fddthm}. 
  We write $\overline Y_n(\thvec) $ as a sum of independent mean zero random variables
(under $P^\om$)  so that we can apply Lindeberg-Feller \cite{durr}:
 \be \overline Y_n(\thvec) =
n^{-1/4}\sum_{i=1}^N\theta_i\big{\{}Y_{n}(t_i,r_i)-E^\om Y_{n}(t_i,r_i)\big{\}}=
 W_n=\sum_{m=-\infty}^{\infty} \bar{U}_m  \label{sum1}\ee
with 
\begin{equation}
\label{umbar}
\bar{U}_m=\, \sum_{i=1}^N \theta_i \Bigl( U_m(t_i,r_i) \,\ind\{m >0\}  
- V_m(t_i,r_i) \, \ind\{m \le 0\} \Bigr), 
\end{equation}
and 
\begin{align}
&\begin{split} U_m(t,r)&=n^{-1/4}\sum_{j=1}^{\eta_0(m)} \mathbf{1}\lbrace X^{m,j}_{\fl{nt}}
\le \fl{nvt}+\fl{r\sqn}\rbrace \\
&\qquad\qquad \qquad\qquad-n^{-1/4} E^\om(\eta_0(m)) P^\om(X^m_{\fl{nt}}\le \fl{nvt}+\fl{r\sqn}\,),\end{split} \label{Udef}\\
&\begin{split}V_m(t,r)&=n^{-1/4}\sum_{j=1}^{\eta_0(m)} \mathbf{1}\lbrace X^{m,j}_{\fl{nt}}
> \fl{nvt}+\fl{r\sqn}\rbrace \\
&\qquad\qquad \qquad\qquad  -n^{-1/4} E^\om(\eta_0(m))  P^\om(X^m_{\fl{nt}}> \fl{nvt}+\fl{r\sqn}\,). 
\end{split}\nn\end{align}
The $n$-dependence is suppressed from the notations $\bar{U}_m$, $U_m(t,r)$ and $V_m(t,r)$. 
The variables $\{\bar{U}_m \}_{m\in\bZ}$ are independent under $P^\om$
because initial occupation variables and walks are independent.   We will also use
repeatedly this formula, a consequence of the independence of $\eta_0$ and the
walks under $P^\om$:   
\be \begin{aligned}
 E^\om[ U_m(t,r)^2 ]&=
n^{-1/2} \Var^\om\biggl(\, \sum_{j=1}^{\eta_0(m)} \mathbf{1}\lbrace X^{m,j}_{\fl{nt}} \le \fl{nvt}+\fl{r\sqn}\,\rbrace
\biggr) \\
&=n^{-1/2}E^\om(\eta_0(m)) P^\om(X^m_{\fl{nt}}\le \fl{nvt}+\fl{r\sqn}\,) P^\om(X^m_{\fl{nt}}> \fl{nvt}+\fl{r\sqn}\,)\\[5pt]
&\qquad\qquad 
+ n^{-1/2}\Var^\om(\eta_0(m)) P^\om(X^m_{\fl{nt}}\le \fl{nvt}+\fl{r\sqn}\,)^2
\end{aligned}\label{varsum}\ee
and the corresponding formula for $V_m(t,r)$.  
  
Let $a(n)\nearrow\infty$ be a sequence that will be determined precisely in the proof.  
Define the finite sum   
\be   W_n^*=\sum_{\abs m\le a(n)\sqn}  \bar{U}_m . \label{sum2}\ee
We observe that the terms $\abs m> a(n)\sqn$ can be discarded from \eqref{sum1}.

\begin{lemma}   $E\abs{W_n-W_n^*}^2\to 0$ as $n\to\infty$.  \label{Slemma}\end{lemma}
\begin{proof} By the mutual independence of occupation variables and  walks under $P^\om$, 
and as eventually
$a(n)>\abs{r_i}$, the task boils down to showing that sums of this type vanish:
\begin{align*}  &E\biggl[\biggl(\, \sum_{m>a(n)\sqn}  U_m(t,r) \biggr)^2\,\biggr] 
=\bE  \sum_{m>a(n)\sqn}  E^\om[ U_m(t,r)^2 ] \\
&\le  n^{-1/2}\,  \bE  \sum_{m>a(n)\sqn}   \bigl[ E^\om(\eta_0(m))
+ \Var^\om(\eta_0(m)) \bigr]  P^\om\{ X^{m,j}_{\fl{nt}} \le \fl{nvt}+\fl{r\sqn}\,\} \\
&\le  Cn^{-1/2}   \sum_{m>a(n)\sqn}  P\{ X_{\fl{nt}} \le \fl{nvt}+\fl{r\sqn}-m\,\}\\
&=C  E\biggl[ \biggl(\frac{X_{\fl{nt}}-\fl{nvt}}{\sqn} -r+a(n)  \biggr)^-\,\biggr].
\end{align*}
Under the averaged measure $P$ the walk $X_s$ is a sum of bounded  i.i.d.\ random
variables, hence by uniform integrability the last line vanishes as $a(n)\nearrow\infty$. 
There is also a term for $m<a(n)\sqn$ involving $V_m(t,r)$ that is handled in the same way.
\end{proof} 
 
The limit $\thvec\cdot \Zvec$ in our goal \eqref{fddlim}  has variance 
\be 
\sigma_\theta^2=\sum_{1\le i,j\le N}  \theta_i\theta_j \Bigl[  
\rho_0\Gamma_1\big((t_i,r_i),(t_j,r_j)\big)+\qvar^2\Gamma_2\big((t_i,r_i),(t_j,r_j)\big) \Bigr]  
\label{sitheta}\ee 
and the two $\Gamma$-terms, defined earlier in \eqref{Ga1} and \eqref{Ga2}, 
 have the following expressions in terms of a standard 1-dimensional
Brownian motion $B_t$:
 \be \begin{aligned}
\Gamma_1\bigl((s,q),(t,r)\bigr)  
&=  \int_{-\infty}^{\infty}\Bigl(  \mP[B_{\sigma^2s}\le q-x]\mP[ B_{\sigma^2t}> r-x] \\
&\qquad\qquad\qquad- \; \mP[B_{\sigma^2s}\le  q-x, B_{\sigma^2t}> r-x] \Bigr)\,dx 
\end{aligned}\label{Ga1a}\ee
and 
 \be \begin{aligned}
\Gamma_2\bigl((s,q),(t,r)\bigr)  
&=
 \int_{0}^\infty \mP[B_{\sigma^2s}\le  q-x]\mP[ B_{\sigma^2t}\le r-x]\,dx\\
 &\qquad\qquad +\; \int_{-\infty}^{0} \mP[B_{\sigma^2s}>  q-x]\mP[ B_{\sigma^2t}> r-x]\,dx. 
\end{aligned}\label{Ga2a}\ee

By Lemma \ref{Slemma},  the desired limit \eqref{fddlim} follows from showing  
\be   E^{\om}(e^{iW_n^*}) 
 \to   e^{-\sigma_\theta^2/2} \quad 
\mbox{ in $\bP$-probability as $n\to\infty$.}\label{goal2}\ee 
 This limit will be achieved by showing that the usual conditions of the Lindeberg-Feller
theorem hold in $\P$-probability:
\be 
\sum_{\abs m\le a(n)\sqn} E^\om(\bar{U}_m^2) \to  \sigma_\theta^2 \label{LF-1}\ee
and 
\be  \sum_{\abs m\le a(n)\sqn}E^\om \big(\, \vert\bar{U}_m\vert^2 
\mathbf{1}\lbrace\vert \bar{U}_m \vert \ge \e  \rbrace \big) \to 0.    \label{LF-2}\ee
The standard Lindeberg-Feller theorem can then be applied to subsequences.
The limits   \eqref{LF-1}--\eqref{LF-2} in $\P$-probability
 imply that every subsequence has a further subsequence
 along which these limits hold for $\P$-almost every $\om$.  Thus along this
further subsequence 
  $W_n^*$ converges weakly  to $\cN(0,\sigma_\theta^2)$ under 
$P^\om$ for $\P$-almost every $\om$. 
So, every subsequence has a further subsequence
 along which the limit \eqref{goal2} holds for $\P$-almost every $\om$. 
 This implies the limit \eqref{goal2}
in $\P$-probability.

We check the negligibility condition  \eqref{LF-2} in the $L^1$ sense. 

\begin{lemma}  
Under assumption \eqref{momass}, 
\be  \lim_{n\to\infty}  \sum_{\abs m\le a(n)\sqn}E\Big[ \vert\bar{U}_m\vert^2 
\mathbf{1}\lbrace\vert \bar{U}_m \vert \ge \e  \rbrace \Big] =0.  \label{lf:tech0}\ee
\end{lemma}
\begin{proof}  First
\begin{align*}
\bar{U}_m^2
&= \biggl(\;\sum_{i=1}^N \theta_i \Big[U_m(t_i,r_i)\m1\{m\ge0\}-V_m(t_i,r_i)\m1\{m<0\} \Big]
\biggr)^2\\
&\le
C \sum_{i=1}^N U_m(t_i,r_i)^2\m1\{m\ge0\}
+ C \sum_{i=1}^N V_m(t_i,r_i)^2\m1\{m<0\}.  
\end{align*}
The arguments for the terms  above are the same.  So take a term from the first sum, 
 let $(t,r)=(t_i,r_i)$, and the task is now 
\be  \lim_{n\to\infty}  
\sum_{m=0}^{a(n)\sqn}   E\bigl[U_m(t,r)^2\ind\{ \abs{\bar{U}_m} \ge \e\}\bigr] 
=0.  \label{lf:tech0.5}  \ee 
   Since 
\[\vert\bar{U}_m\vert \le C{n^{-{1}/{4}}}\bigl[ \eta_0(m) + E^\om(\eta(m))\bigr]  \] 
and by
adjusting $\e$,  limit \eqref{lf:tech0.5} follows if we can show the limit for these sums: 
\be \begin{aligned}
&\sum_{m=0}^{a(n)\sqn}   E\bigl[U_m(t,r)^2\ind\{\eta_0(m)> n^{1/4}\e\}\bigr] \\[5pt]
&\qquad + \;  \sum_{m=0}^{a(n)\sqn}  E\bigl[U_m(t,r)^2\ind\{E^\om(\eta_0(m))> n^{1/4}\e\}\bigr].
\end{aligned} \label{lf:tech1}\ee
 
Abbreviate 
\[ A_{m} =\lbrace X^{m}_{nt}\le \fl{nvt}+\fl{r\sqn}\,\rbrace.\]
 The    terms of the second sum in \eqref{lf:tech1}  develop as follows, using \eqref{varsum}, the
independence of $\wlev_{-\infty,-1}$ and $\wlev_{0,\infty}$, and the shift invariance:
\begin{align*}
&\E\bigl[ E^\om(U_m(t,r)^2) \ind\{E^\om(\eta_0(m))> n^{1/4}\e\}\bigr]\\[4pt]
&\le  n^{-1/2} \E\bigl[  \bigl(  E^\om(\eta_0(m))
+ \Var^\om(\eta_0(m))  \bigr)  \ind\{E^\om(\eta_0(m))> n^{1/4}\e\}\bigr] P(A_{m})\\[4pt]
&= n^{-1/2} \E\bigl[  \bigl(  E^\om(\eta_0(0))
+ \Var^\om(\eta_0(0))  \bigr)  \ind\{E^\om(\eta_0(0))> n^{1/4}\e\}\bigr] P(A_{m}).
\end{align*}
Since the averaged walk is a walk with bounded i.i.d.~steps, 
\be  \sum_{m=0}^{a(n)\sqn} P(A_{m}) \le E\bigl[(X_{\fl{nt}}-\fl{nvt}-\fl{r\sqrt n}\,)^-\bigr]\le C(n^{1/2}+1). 
\label{lf:tech3}\ee
Thus 
\begin{align*}
&\sum_{m=0}^{a(n)\sqn}  E\bigl[U_m(t,r)^2\ind\{E^\om(\eta_0(m))> n^{1/4}\e\}\bigr]\\
&\qquad \le 
C  \E\bigl[  \bigl(  E^\om(\eta_0(0))
+ \Var^\om(\eta_0(0))  \bigr)  \ind\{E^\om(\eta_0(0))> n^{1/4}\e\}\bigr].  \end{align*}
 The last line   vanishes as $n\to\infty$ by dominated convergence, by assumption
 \eqref{momass}. 
 
 For the first sum in \eqref{lf:tech1}   first take quenched expectation   of the walks 
while conditioning on  $\eta_0$, to get the bound
 \[  E^\om_{\eta_0}[U_m(t,r)^2] \le 2n^{-1/2} P^\om(A_{m}) \bigl[\eta_0(m)^2 
+  E^\om(\eta_0(m))^2\bigr].  \]
Using again the independence of $\wlev_{-\infty,-1}$ and $\wlev_{0,\infty}$,  shift-invariance, 
and \eqref{lf:tech3}, 
\begin{align*} &\sum_{m=0}^{a(n)\sqn}  E\bigl[U_m(t,r)^2\ind\{\eta_0(m)> n^{1/4}\e\}\bigr]\\
 &  \le Cn^{-1/2}  \sum_{m=0}^{a(n)\sqn} P(A_{m})
\cdot  E\bigl[ \bigl(\eta_0(0)^2+  E^\om(\eta_0(0))^2 \bigr) \ind\{\eta_0(0)> n^{1/4}\e\}\bigr]\\
&\le C  E\bigl[ \bigl(\eta_0(0)^2+  E^\om(\eta_0(0))^2 \bigr) \ind\{\eta_0(0)> n^{1/4}\e\}\bigr]
 \end{align*}
 The last line   vanishes as $n\to\infty$ by dominated convergence, by assumption
 \eqref{momass}. 
\end{proof} 

We turn to checking \eqref{LF-1}.  
\be \begin{aligned} 
& \sum_{\abs m\le a(n)\sqn} E^{\om}\big[\bar{U}_m^2 \big]\\
&=
\sum_{1 \le i, j \le N}   \theta_i\theta_j   \sum_{\abs m\le a(n)\sqn}
\Bigl[  \ind_{\{m>0\}}E^\w\bigl(U_m(t_i,r_i)U_m(t_j,r_j)\bigr)\\[3pt]
  &\qquad\qquad +\;   
 \ind_{\{m\le 0\}}E^\w\bigl(V_m(t_i,r_i)V_m(t_j,r_j)\bigr)\Bigr].
\nn\end{aligned}\label{vpfindim1}\ee
Each quenched expectation  of a product of two mean zero random
variables is handled in the manner of \eqref{varsum} 
 that we demonstrate with the second expectation:
\begin{align*} 
&E^\w\bigl(V_m(t_i,r_i)V_m(t_j,r_j)\bigr)\\
&= n^{-1/2}
 \Cov^\om\biggl(\, \sum_{k=1}^{\eta_0(m)} \mathbf{1}\lbrace X^{m,k}_{\fl{nt_i}} > \fl{nvt_i}+r_i\sqn\,\rbrace,
\sum_{\ell=1}^{\eta_0(m)} \mathbf{1}\lbrace X^{m,\ell}_{\fl{nt_j}} > \fl{nvt_j}+r_j\sqn\,\rbrace 
\biggr) \\
&=n^{-1/2} E^\om(\eta_0(m)) \Bigl[ P^\om(X^m_{\fl{nt_i}}> \fl{nvt_i}+r_i\sqn, \, 
X^m_{\fl{nt_j}}> \fl{nvt_j}+r_j\sqn\,)\\[4pt]
&\qquad  \qquad\qquad
 -\; P^\om(X^m_{\fl{nt_i}}> \fl{nvt_i}+r_i\sqn\,) P^\om(X^m_{\fl{nt_j}}> \fl{nvt_j}+r_j\sqn\,)\Bigr]\\[5pt]
&\qquad 
+n^{-1/2} \Var^\om(\eta_0(m)) P^\om(X^m_{\fl{nt_i}}> \fl{nvt_i}+r_i\sqn\,)
P^\om(X^m_{\fl{nt_j}}> \fl{nvt_j}+r_j\sqn\,).
 \end{align*}
 
 After some rearranging of the resulting probabilities, we arrive at 
\be  \begin{aligned}
& \sum_{\abs m\le a(n)\sqn} E^{\om}\big[\bar{U}_m^2 \big] \\
 &= n^{-1/2} \sum_{1\le i,j\le N}\theta_i\theta_j \,\biggl[ \;
 \sum_{\abs m\le a(n)\sqn}  E^\w(\eta_0(m)) \\
 &\qquad\quad \times \Bigl\{
P^\w(X^{m}_{\fl{nt_i}}\le \fl{nvt_i}+\fl{r_i\sqrt{n}}\,)P^\w( X^{m}_{\fl{nt_j}}> \fl{nvt_j}+\fl{r_j\sqrt{n}} \,)
\\
& \qquad \qquad\qquad \qquad
 -\; P^\w(X^{m}_{\fl{nt_i}}\le \fl{nvt_i}+\fl{r_i\sqrt{n}},\, X^{m}_{\fl{nt_j}}>  \fl{nvt_j}+\fl{r_j\sqrt{n}} \,)\Bigr\} 
 \\[4pt]
&+ \sum_{\abs m\le a(n)\sqn}  \Var^\w(\eta_0(m)) \\
&\qquad \quad \ \times \Bigl\{ 
 \ind_{\{m>0\}} 
 P^\w(X^{m}_{\fl{nt_i}}\le \fl{nvt_i}+\fl{r_i\sqrt{n}}\,)
   P^\w( X^{m}_{\fl{nt_j}}\le \fl{nvt_j}+\fl{r_j\sqrt{n}} \,) 
 \\[6pt]
 & \qquad \qquad  + \ind_{\{m\le 0\}} 
 P^\w(X^{m}_{\fl{nt_i}}> \fl{nvt_i}+\fl{r_i\sqrt{n}}\,)P^\w( X^{m}_{\fl{nt_j}}> \fl{nvt_j}+\fl{r_j\sqrt{n}} \,) \,\Bigr\}\;\biggr].    \end{aligned}   \label{fd6}\ee
The terms above have been arranged so that the sums match up with the integrals 
in \eqref{sitheta}--\eqref{Ga2a}.  Limit \eqref{LF-1} is now proved by showing that, term
by term, the sums above converge to the integrals.  In each case the argument is the
same.   We illustrate the case of the sum of the first term with the factor $ \Var^\w(\eta_0(m))$
in front.  To simplify notation we let $((s,q),(t,r))=((t_i,r_i),(t_j,r_j))$.
 In other words,  we show this convergence in $\bP$-probability:   
\be\begin{aligned}
&S_0(n)\equiv n^{-1/2} \sum_{ 0 < m\le a(n)\sqn}  \Var^\w(\eta_0(m))  
  P^\w(X^{m}_{\fl{ns}}\le \fl{nvs}+\fl{q\sqrt{n}}\,)\\
&\qquad \qquad \qquad \qquad\qquad \qquad
 \ \times   P^\w( X^{m}_{\fl{nt}}\le \fl{nvt}+\fl{r\sqrt{n}} \,)\\
&\underset{n\to\infty}\longrightarrow \; 
\qvar^2 \int_{0}^\infty \mP[B_{\sigma^2s}\le  q-x]\mP[ B_{\sigma^2t}\le r-x]\,dx\equiv I.
\end{aligned}\label{fdgoal6}\ee
  The proof of $S_0(n)\overset{\bP}\to I$ is divided  into two lemmas. 
Let 
\be\begin{aligned}
S_1(n)= &n^{-1/2} \sum_{ 0 < m\le a(n)\sqn}  \Var^\w(\eta_0(m))  \\
&\qquad \qquad \qquad \times
\mP\Bigl(B_{\sigma^2s}\le  q-\frac{m}{\sqn}\,\Bigr)\mP\Bigl( B_{\sigma^2t}\le r-\frac{m}{\sqn}\,\Bigr). 
 \end{aligned}\label{S1(n)}\ee

\begin{lemma}   $\ddd\lim_{n\to\infty} \E\abs{S_0(n)-S_1(n)}=0.$
\end{lemma} 
 \begin{proof}  By the quenched central limit theorem for space-time RWRE
 \cite{rass-sepp-05},  for each $x\in\bR$ the limit 
 \[  
 P^\om (X_{\fl{ns}}\le \fl{nvs}+\fl{x\sqrt{n}}\,)  \to \mP( B_{\sigma^2s}\le x) \]
 holds for $\P$-a.e.\ $\om$.  Since these are distribution functions (monotone and 
 between $0$ and $1$) with a continuous limit  the convergence is uniform in $x$.  Set 
 \begin{align*}
 D_n(\om)=\sup_{x,y\in\bR} \; \bigl\lvert    P^\w(X_{\fl{ns}}\le \fl{nvs}+\fl{x\sqrt{n}}\,)&
  P^\w( X_{\fl{nt}}\le \fl{nvt}+\fl{y\sqrt{n}} \,) \\
  &\qquad   -\;  \mP( B_{\sigma^2s}\le x)  \mP( B_{\sigma^2t}\le y)
\, \bigr\rvert 
\end{align*}
and then $D_n(\om)\to 0$ $\P$-a.s.   By shift-invariance 
\begin{align}
&\E\abs{S_0(n)-S_1(n)} \le  n^{-1/2} \sum_{ 0 < m\le a(n)\sqn}  
\E \Var^{T_{m,0}\w}(\eta_0(0)) \nn\\
&\qquad\times  \Bigl\lvert 
  P^{T_{m,0}\w}(X_{\fl{ns}}\le \fl{nvs}+\fl{q\sqrt{n}}-m\,)  P^{T_{m,0}\w}( X_{\fl{nt}}\le \fl{nvt}+\fl{r\sqrt{n}} -m\,)\nn\\[4pt]
&\qquad\qquad \qquad\qquad- \; \mP\Bigl(B_{\sigma^2s}\le  q-\frac{m}{\sqn}\,\Bigr)
\mP\Bigl( B_{\sigma^2t}\le r-\frac{m}{\sqn}\,\Bigr)\,\Bigr\rvert\nn \\
&\le n^{-1/2} \sum_{ 0 < m\le a(n)\sqn}  
\E\bigl[ \Var^{T_{m,0}\w}(\eta_0(0)) D_n(T_{m,0}\w)\bigr]\nn \\
&\le  2a(n) \E\bigl[  \Var^{\w}(\eta_0(0)) D_n(\w)\bigr].  \label{tech9} 
\end{align}
Moment assumption \eqref{momass} and dominated convergence guarantee
that \[\E\bigl[  \Var^{\w}(\eta_0(0)) D_n(\w)\bigr]\longrightarrow 0.\]  Thus we can take 
\be
a(n)=  \Bigl( \, \sup_{k:k\ge n}  \E\bigl[ \Var^{\w}(\eta_0(0)) D_k(\w)\bigr]\,  \Bigr)^{-1/2}  
\label{a(n)}\ee
to have $a(n)\nearrow \infty$ while still line \eqref{tech9} vanishes as $n\to\infty$. 
\end{proof}

The choice of $a(n)$ made above depends on $s,t$ but that is not problematic 
since we have only finitely many time points $t_i$ to handle. 

\begin{lemma}   $\ddd\lim_{n\to\infty} \E\abs{S_1(n)-I}=0.$
 \end{lemma} 
 \begin{proof}  First we discard tails of the sum and integral. 
Given $\e>0$, we can choose a large enough $c<\infty$  such that 
\begin{align*}
S_1^*(n)= &n^{-1/2} \sum_{ 0 < m\le c\sqn}  \Var^\w(\eta_0(m))  \\
&\qquad \qquad \qquad \times
\mP\Bigl(B_{\sigma^2s}\le  q-\frac{m}{\sqn}\,\Bigr)\mP\Bigl( B_{\sigma^2t}\le r-\frac{m}{\sqn}\,\Bigr) 
 \end{align*}
 satisfies $\E\abs{S_1(n)-S_1^*(n)}\le \e$, and so that 
\[    I^*= \qvar^2 \int_{0}^c \mP[B_{\sigma^2s}\le  q-x]\mP[ B_{\sigma^2t}\le r-x]\,dx  \]
satisfies $I-I^*\le \e$.   
   Thus it suffices to prove 
 $S_1^*(n)\to I^*$.  
 
 Next, since the Gaussian distribution functions are Lipschitz continuous, 
 \begin{align*}  
&S_1^*(n)- I^* = n^{-1/2} \sum_{ 0 < m\le c\sqn}  
\bigl[ \Var^\w(\eta_0(m)) - \qvar^2\bigr]   \\
&\qquad \qquad   \times
\mP\Bigl(B_{\sigma^2s}\le  q-\frac{m}{\sqn}\,\Bigr)\mP\Bigl( B_{\sigma^2t}\le r-\frac{m}{\sqn}\,\Bigr) 
+ O(n^{-1/2}). 
\end{align*}
Introduce an intermediate scale $1<<L<<\sqn$ and use again  
 Lipschitz continuity of the probabilities: 
 \begin{align*}  
&S_1^*(n)- I^* =  \frac{L}{n^{1/2}} \sum_{ 0\le j \le \frac{ c\sqn}L-1  }  
\biggl( \frac1{L}  \sum_{m=jL+1}^{(j+1)L} \Var^\w(\eta_0(m)) - \qvar^2\biggr) \\[3pt]
&\quad   \times
\biggl\{ \mP\Bigl(B_{\sigma^2s}\le  q-\frac{jL}{\sqn}\,\Bigr)
\mP\Bigl( B_{\sigma^2t}\le r-\frac{jL}{\sqn}\,\Bigr) 
+O\Bigl(\frac{L}{\sqn}\Bigr) \biggr\} +\frac{R_n}{\sqn} + O(n^{-1/2}). 
\end{align*}
The error term $R_n$ consists of order $L$ terms bounded by 
$\abs{ \Var^\w(\eta_0(m)) - \qvar^2}$ that appear because the 
collection of summation
intervals $(jL,(j+1)L]$ may not exactly cover the original summation interval
$0 < m\le c\sqn$.   It satisfies $\E R_n \le CL$. 
Finally, bounding the probabilities crudely by 1 and by shift-invariance, 
 \begin{align*}  
&\E\abs{S_1^*(n)- I^*} \le  C 
\E \biggl\lvert  \frac1{L}  \sum_{m=1}^{L}  \Var^\w(\eta_0(m)) 
- \qvar^2\biggr\rvert   + O(Ln^{-1/2}).  
\end{align*}
This vanishes as we let first $n\to\infty$ and then $L\to\infty$ and apply the 
$L^1$ ergodic theorem.  
 \end{proof}
 
 Limit \eqref{fdgoal6} has now been verified.  All  terms in \eqref{fd6} are treated
 the same way to show that they converge, in $L^1(\P)$ and therefore in 
 $\P$-probability, to the corresponding  integrals in  \eqref{sitheta}--\eqref{Ga2a}. 
 This verifies limit \eqref{LF-1}. Since both \eqref{LF-1} and \eqref{LF-2} have been 
 checked,  the Gaussian limit in \eqref{goal2} has been proved, as explained in the
 paragraph following  \eqref{LF-2}.    The proof of Proposition 
 \ref{fddprop} and thereby also the proof of Theorem \ref{fddthm} are complete. \qed

\section{The quenched mean process} \label{var-covar} 

We now prove Theorems \ref{th:SHE} and \ref{th:fBM}.
We will use a simplified notation for the quenched jump probabilities:   $\w_{x,n}=\w_{x,n}(1)$ and   $\w'_{x,n}=\w_{x,n}(0)=1-\w_{x,n}(1)$.
Note that when the steps are $0$ and $1$ we have $v=p(1)=\E\w_{0,0}$.
Potential kernel $\ahom$ can be easily computed from equations \eqref{a:eqn} and seen to equal $\ahom(x)=\frac{\abs{x}}{2v(1-v)}$.
Recall that $\alpha=\bE\w_{0,0}\w'_{0,0}$. 
Then formula \eqref{beta2} gives 
	\[\beta=\frac{\alpha}{v(1-v)}\,.\]

 \begin{proof}[Proof of Theorem \ref{th:SHE}]
Define 
	\[H_n(x)=E^\w\Bigl[\sum_{y>0}\sum_{j=1}^{\eta_0(y)}\ind\{X_n^{y,j}\le x\}-\sum_{y\le0}\sum_{j=1}^{\eta_0(y)}\ind\{X_n^{y,j}> x\}\Bigr].\]
Then $Y_n(t,r)=H_{\fl{nt}}(\fl{nvt}+\fl{r\sqrt n})$.
Compute
	\begin{align*}
	H_{n+1}(x)
	&=E^\w\Bigl[\sum_{y>0}\sum_{j=1}^{\eta_0(y)}\ind\{X_n^{y,j}\le x-1\}\Bigl]+E^\w\Bigl[\sum_{y>0}\sum_{j=1}^{\eta_0(y)}\ind\{X_n^{y,j}=x\}\Bigl]\w'_{x,n}\\
	&\qquad-\sum_{y\le0}\sum_{j=1}^{\eta_0(y)}\ind\{X_n^{y,j}> x\}\Bigr]-\sum_{y\le0}\sum_{j=1}^{\eta_0(y)}\ind\{X_n^{y,j}=x\}\Bigr]\w_{x,n}.
	\end{align*}
Also,
	\begin{align*}
	&\w_{x,n} H_n(x-1)+\w'_{x,n} H_n(x)\\
	&\qquad=E^\w\Bigl[\sum_{y>0}\sum_{j=1}^{\eta_0(y)}\ind\{X_n^{y,j}\le x-1\}\Bigl]\w_{x,n}-\sum_{y\le0}\sum_{j=1}^{\eta_0(y)}\ind\{X_n^{y,j}> x-1\}\Bigr]\w_{x,n}\\
	&\qquad\qquad+E^\w\Bigl[\sum_{y>0}\sum_{j=1}^{\eta_0(y)}\ind\{X_n^{y,j}\le x\}\Bigl]\w'_{x,n}-\sum_{y\le0}\sum_{j=1}^{\eta_0(y)}\ind\{X_n^{y,j}> x\}\Bigr]\w'_{x,n}.
	\end{align*}
Taking the difference of the two expressions one finds that
	\[H_{n+1}(x)=\w_{x,n} H_n(x-1)+\w'_{x,n} H_n(x).\]
In other words, $H$ is the random average process introduced by Ferrari and Fontes \cite{ferr-font-rap}.
The initial conditions are given by
	\[H_0(x)=\begin{cases}0&\text{if }x=0,\\ \displaystyle\sum_{y=1}^x E^\w\eta_0(y)&\text{if }x>0,\quad\text{and}\\[12pt]\displaystyle\sum_{y=x+1}^0 E^\w\eta_0(y)&\text{if }x<0.\end{cases}\]
The claim now follows by applying \cite[Thm.~4.1]{sepp-10-ens} and the characterization on page 13 of \cite{sepp-10-ens}. 
(\cite[Thm.~4.1]{sepp-10-ens} as reproduced from \cite[Thm.~2.1]{bala-rass-sepp} where the limiting stochastic heat equation
  is slightly altered because the process studied was $H_{\fl{nt}}(\fl{nvt}+\fl{r\sqrt n})-H_0(\fl{r\sqrt n})$.)
\end{proof}
 
\begin{proof}[Proof of Theorem \ref{th:fBM}]
Now, we have $\beta=\alpha/(v(1-v))=4\alpha$. We will write $p_{x,y}^k$ for the $k$-step averaged transition. 
%
For $t\ge0$ define
 \begin{align*}
 Y(t)= \sum_{x>0}  \sum_{j=1}^{\eta_0(x)} \mathbf{1}\lbrace X^{x,j}_{\fl{t}}\le \fl{vt}\,\rbrace
- \sum_{x\le 0} \sum_{j=1}^{\eta_0(x)} \mathbf{1}\lbrace X^{x,j}_{\fl{t}}> \fl{vt}\,\rbrace.
\end{align*}  

By stationarity
	\[E^\w Y_n(t,0)-E^\w Y_n(s,0)=E^\w Y(nt)-E^\w Y(ns)\]
has the same distribution as the $E^\w$-mean of 
\[Y'=\sum_{x>0}  \sum_{j=1}^{\eta_0(x)} \mathbf{1}\lbrace X^{x,j}_{\fl{nt}-\fl{ns}}\le \fl{nvt}-\fl{nvs}\,\rbrace
- \sum_{x\le 0} \sum_{j=1}^{\eta_0(x)} \mathbf{1}\lbrace X^{x,j}_{\fl{nt}-\fl{ns}}> \fl{nvt}-\fl{nvs}\,\rbrace.\]
The difference $\abs{Y'-Y(\fl{nt}-\fl{ns})}$ is bounded by the number of particles that are at time $\fl{nt}-\fl{ns}$
between $\fl{nvt}-\fl{nvs}$ and $\fl{(\fl{nt}-\fl{ns})v}$. Since $\abs{\fl{nvt}-\fl{nvs}-\fl{(\fl{nt}-\fl{ns})v}}\le2$ we are talking about at most $5$ sites and,
consequently, $\E[\abs{E^\w Y'-E^\w Y(\fl{nt}-\fl{ns})}]\le 5\E[f]=5$. A similar reasoning gives a bound on $\E[\abs{Y(nt)-Y(\fl{nt})}]$ and $\E[\abs{Y(ns)-Y(\fl{ns})}]$.
Therefore,
	\begin{align*}
	&\lim_{n\to\infty}\frac1{\sqrt{n}}\Vvv\bigl(E^\w Y(\fl{nt}-\fl{ns})\bigr)
	=\lim_{n\to\infty}\frac1{\sqrt{n}}\Vvv\bigl(E^\w Y_n(t,0)-E^\w Y_n(s,0)\bigr)\\
	&=\lim_{n\to\infty}\frac1{\sqrt{n}}\Bigl[\Vvv\bigl(E^\w Y(\fl{nt})\bigr)+\Vvv\bigl(E^\w Y(\fl{ns})\bigr)-2\Cvv\bigl(E^\w Y_n(s,0),E^\w Y_n(t,0)\bigr)\Bigr].
	\end{align*}
Hence, it is enough to prove that 
\[ \lim_{n\to\infty} \frac1{\sqrt{n}} \Vvv\bigl(E^\w Y(n)\bigr)=\frac{1}{\sqrt{2\pi}}\bigl( \tfrac14 \alpha^{-1}-1\bigr). \]  
Since
	\[E^\w Y(2n+1)-E^\w Y(2n)=-f(T_{n,2n}\w)\w_{n,2n}\]
we see that it is enough to  prove the above limit along the subsequence of even integers.

Let
\begin{align*}
h(\w)=f(T_{1,0}\w)\w'_{1,0}\w'_{1,1}-f(\w)\w_{0,0}\w_{1,1}.
\end{align*}
Then $\bE(h)=0$, $\bE(h^2) = \frac1{8\alpha}-\frac12$ (here we use Corollary \ref{cor:covf} and $p_0=p_1=1/2$), and
\[E^\w Y(2n+2)-E^\w Y(2n)=h(T_{n,2n}\w).\]   
Let $c_0=\Vvv(f)=\beta^{-1}-1$.  To compute $\bE h(\w)h(T_{n,2n}\w)$ write 
\begin{align*}
h(\w)=(f(T_{1,0}\w)-1)\w'_{1,0}\w'_{1,1}-(f(\w)-1)\w_{0,0}\w_{1,1}
+ \w'_{1,0}\w'_{1,1}-\w_{0,0}\w_{1,1}
\end{align*}
and 
\begin{align*}
h(T_{n,2n}\w)&=\sum_{-n+1\le x\le n+1}(f(T_{x,0}\w)-1) \pi_{0,2n}(x,n+1)\w'_{n+1,2n}\w'_{n+1,2n+1}\\
&\qquad-\sum_{-n\le y\le n}(f(T_{y,0}\w)-1) \pi_{0,2n}(y,n)\w_{n,2n}\w_{n+1,2n+1}\\
&\qquad+ \sum_{-n+1\le x\le n+1} \pi_{0,2n}(x,n+1)\w'_{n+1,2n}\w'_{n+1,2n+1}\\
&\qquad-\sum_{-n\le y\le n}  \pi_{0,2n}(y,n)\w_{n,2n}\w_{n+1,2n+1}.  
\end{align*}
Due to $\kS_{-\infty,-1}$-measurability the $f$-terms are independent of the $\w$'s. 
Also, distinct shifts are uncorrelated by Corollary \ref{cor:covf}.  
Multiplying
these terms together and separating the expectations of the factors on levels $2n$ and 
$2n+1$  leads to 
\begin{align}
\bE h(\w)h(T_{n,2n}\w) &= \tfrac14c_0 \bE \pi_{0,2n}(1,n+1)\w'_{1,0}\w'_{1,1}  \label{t1}\\[5pt]
&\qquad- \tfrac14c_0  \bE \pi_{0,2n}(0,n+1) \w_{0,0}\w_{1,1}  \label{t2} \\[5pt]
&\qquad- \tfrac14c_0 \bE \pi_{0,2n}(1,n)\w'_{1,0}\w'_{1,1}  \label{t3}\\[5pt]
&\qquad+ \tfrac14c_0  \bE \pi_{0,2n}(0,n) \w_{0,0}\w_{1,1}  \label{t4} \\[5pt]
&\qquad+ \tfrac14\sum_{-n+1\le x\le n+1} \bE\pi_{0,2n}(x,n+1) \bigl(\w'_{1,0}\w'_{1,1}
-\w_{0,0}\w_{1,1}\bigr)  \label{t5} \\[5pt]
&\qquad-\tfrac14\sum_{-n\le y\le n}  \bE\pi_{0,2n}(y,n) \bigl(\w'_{1,0}\w'_{1,1}
-\w_{0,0}\w_{1,1}\bigr) .   \label{t6} 
\end{align}
Thinking through the possible jumps shows that the terms in \eqref{t5} and \eqref{t6} 
survive only for $x,y\in\{0,1\}$.  And some of these terms can be combined with the ones
above.  This gives 
\begin{align}
\bE h(\w)h(T_{n,2n}\w) &= \tfrac14(c_0+1) \bE \pi_{0,2n}(1,n+1)\w'_{1,0}\w'_{1,1}  \label{t2.1}\\[5pt]
&\qquad- \tfrac14(c_0+1)  \bE \pi_{0,2n}(0,n+1) \w_{0,0}\w_{1,1}  \label{t2.2} \\[5pt]
&\qquad- \tfrac14(c_0+1) \bE \pi_{0,2n}(1,n)\w'_{1,0}\w'_{1,1}  \label{t2.3}\\[5pt]
&\qquad+ \tfrac14(c_0+1)  \bE \pi_{0,2n}(0,n) \w_{0,0}\w_{1,1}  \label{t2.4} \\[5pt]
&\qquad+ \tfrac14  \bE\bigl[\pi_{0,2n}(0,n+1) \w'_{1,0}\w'_{1,1}
-\pi_{0,2n}(1,n+1)\w_{0,0}\w_{1,1} \bigr]  \label{t2.5} \\[5pt]
&\qquad-\tfrac14   \bE\bigl[\pi_{0,2n}(0,n) \w'_{1,0}\w'_{1,1}
-\pi_{0,2n}(1,n)\w_{0,0}\w_{1,1}  \bigr].   \label{t2.6} 
\end{align}
Now transform each term. For example, term \eqref{t2.1} becomes
\begin{align*}
&\text{\eqref{t2.1}} \\
&= \tfrac14(c_0+1) \bE \Bigl[\bigl(\w'_{1,0}\w'_{1,1}p_{1,n+1}^{2n-2}
+ \w'_{1,0}\w_{1,1}p_{2,n+1}^{2n-2} + \w_{1,0}\w'_{2,1}p_{2,n+1}^{2n-2} 
+\w_{1,0}\w_{1,1}p_{3,n+1}^{2n-2} \bigr)\,\w'_{1,0}\w'_{1,1} \Bigr]\\
&= \tfrac14(c_0+1)  \Bigl\{  (\tfrac12-\alpha)^2 p_{1,n+1}^{2n-2}
+ (\tfrac34\alpha-\alpha^2)  p_{2,n+1}^{2n-2} 
+ \tfrac14\alpha p_{3,n+1}^{2n-2} \Bigr\}. 
\end{align*}
After these steps we get  
\begin{align*}
&\bE h(\w)h(T_{n,2n}\w) \\
&\qquad= \tfrac14(c_0+1)\Bigl[-\tfrac14\alpha (p_{0,n+1}^{2n-2} +p_{0,n-3}^{2n-2})\\
&\qquad \qquad \qquad \qquad +(2\alpha^2-\tfrac32\alpha+\tfrac14) (p_{0,n}^{2n-2}+p_{0,n-2}^{2n-2}) 
+ (-4\alpha^2+\tfrac72\alpha-\tfrac12) p_{0,n-1}^{2n-2}  \Bigr]\\[4pt]
&\qquad\qquad\qquad\qquad +\; \tfrac14\Bigl[\tfrac1{16} (p_{0,n+1}^{2n-2} +p_{0,n-3}^{2n-2}) 
 +(\tfrac18-\tfrac12\alpha) (p_{0,n}^{2n-2}+p_{0,n-2}^{2n-2}) 
+ (\alpha-\tfrac38) p_{0,n-1}^{2n-2}  \Bigr].
\end{align*}
Letting $X_k$ denote the (averaged) Markov chain with transition $p_{x,y}$,  introduce  $Z_k=X_{2k}-k$
with transition $r_{0,0}=1/2$,  $r_{0,\pm 1}=1/4$.  For higher order transitions
$p^{2n-2}_{0,n+i}=r^{n-1}_{i+1}$.  Replace the $p$'s with $r$'s and combine them using
symmetry:  $r^{n-1}_2=r^{n-1}_{-2}$, etc. Then
\begin{align}
\bE h(\w)h(T_{n,2n}\w) &= \tfrac14(c_0+1)\Bigl[\tfrac12\alpha (r_{0}^{n-1} -r_{2}^{n-1} )
-(4\alpha^2-3\alpha+\tfrac12) (r_{0}^{n-1} -r_{1}^{n-1} )  
 \Bigr]\notag\\[4pt] 
 &\qquad +\; \tfrac14\Bigl[ -\tfrac18 (r_{0}^{n-1} -r_{2}^{n-1} )  
 + (\alpha-\tfrac14) (r_{0}^{n-1} -r_{1}^{n-1} )  \Bigr]\notag\\[4pt]
 &= \tfrac14(\tfrac12-\tfrac18 \alpha^{-1})(r_{0}^{n-1} -r_{1}^{n-1} )\label{h-cov}
\end{align}
where in the last step we used $c_0+1=(4\alpha)^{-1}$.

Use the potential kernel $a^Z$ of the $r$-walk:  the variance is $1/2$ so $a^Z(x)=2\abs{x}$. 
From \cite{spitzer} and symmetry,   $a^Z(x)=\lim_{m\to\infty} a^Z_m(x)$ with 
\begin{align}\label{aZform}
  a^Z_m(x) =    \sum_{k=0}^m (r^k_{0,0}-r^k_{x,0}) =   \sum_{k=0}^m (r^k_{0}-r^k_{x}).
\end{align}
Then
\begin{align}
\Vvv\bigl(E^\w Y(2n)\bigr)&=\Vvv\biggl[\;\sum_{k=0}^{n-1}  h(T_{k,2k}\w)  \biggr]\nn\\
&=n\bE(h^2) + 2\sum_{k=1}^{n-1} (n-k) \bE h(\w)h(T_{k,2k}\w)\nn\\
&=\bigl( \tfrac1{16}\alpha^{-1}-\tfrac14\bigr)\biggl[ 2n  
-   \sum_{k=1}^{n-1} (n-k)  (r_{0}^{k-1} -r_{1}^{k-1} )    \biggr] \nn\\
&=\bigl( \tfrac1{16}\alpha^{-1}-\tfrac14\bigr)\biggl[ 2n  
-   \sum_{j=1}^{n-1}  \sum_{k=1}^{j}   (r_{0}^{k-1} -r_{1}^{k-1} )    \biggr] \nn\\
&=\bigl( \tfrac1{16}\alpha^{-1}-\tfrac14\bigr)\biggl[ na^Z(1)  -  \sum_{j=1}^{n-1}   a^Z_{j-1}(1) 
\biggr]\biggr] \nn\\
&=\bigl( \tfrac1{16}\alpha^{-1}-\tfrac14\bigr)\biggl[ a^Z(1)  +  \sum_{j=1}^{n-1}  (a^Z(1) - a^Z_{j-1}(1) ) \biggr].
\label{var-3}\end{align}
Let us look at $a^Z_k(x)$.  The characteristic function is 
\[ \zeta(\theta)=\sum_x r_x e^{ix\theta}= \tfrac12(1+\cos\theta). \]
By symmetry:
\begin{align*}
a^Z_m(x)&= \sum_{k=0}^m (r^k_{0}-r^k_{x}) =  \sum_{k=0}^m\frac1{2\pi}\int_{-\pi}^\pi
\bigl(\zeta^k(\theta)- e^{-ix\theta} \zeta^k(\theta)\bigr)\,d\theta \\
&=\frac1{2\pi}\int_{-\pi}^\pi \frac{1-\cos x\theta}{1-\zeta(\theta)}
\bigl(1- \zeta^{m+1}(\theta)\bigr)\,d\theta \\
&=\frac2{2\pi}\int_{-\pi}^\pi \frac{1-\cos x\theta}{1-\cos\theta}
\bigl(1- \zeta^{m+1}(\theta)\bigr)\,d\theta \\
&=2\abs{x} -\frac1{\pi}\int_{-\pi}^\pi \frac{1-\cos x\theta}{1-\cos\theta}
\biggl( \frac{1+\cos\theta}2\biggr)^{m+1}\,d\theta.
\end{align*}
The value of the first  integral above is on p.~61 in \cite{spitzer}.  But actually we only need
$x=1$: 
\begin{align*}
a^Z_m(1)&= a^Z(1) -\frac1{\pi}\int_{-\pi}^\pi  
\biggl( \frac{1+\cos\theta}2\biggr)^{m+1}\,d\theta
=a^Z(1) -\frac2{\pi}\int_{-\pi/2}^{\pi/2}  
 \cos^{2m+2}x\,dx\\
&= a^Z(1) -2\prod_{\ell=1}^{m+1}  \biggl(1-\frac1{2\ell}\biggr)
= a^Z(1) -\frac{2}{(m+1)!}\prod_{\ell=1}^{m+1} (\ell-\tfrac1{2}).
\end{align*}
Put this back into \eqref{var-3}: 
\begin{align*}
\Vvv\bigl(E^\w Y(2n) \bigr)&= 
\bigl( \tfrac1{16}\alpha^{-1}-\tfrac14\bigr)\biggl[ a^Z(1)  +
  \sum_{j=1}^{n-1}  \frac{2}{j!}\prod_{\ell=1}^{j} (\ell-\tfrac1{2})    \biggr] \\
&=\bigl( \tfrac1{8}\alpha^{-1}-\tfrac12\bigr)\biggl[1  +
  \sum_{j=1}^{n-1}    \frac{j^{-1/2}}{\Gamma_j(-1/2) \cdot(-1/2)}\biggr]. 
 \end{align*}
Above we used the definition
\[   \Gamma_m(x)=\frac{m!\, m^x}{x(x+1)\dotsm(x+m)}.  \]
According to p.~461 of \cite{stro-ca},  $\Gamma_m(x)\to \Gamma(x)$
 for $x\notin\bZ_-$.  Plugging back into the above:
\begin{align*}
\Vvv\bigl(E^\w Y(2n)\bigr)&=  \bigl( \tfrac1{8}\alpha^{-1}-\tfrac12\bigr)4\sqrt n\biggl[\frac1{4\sqrt n}  +
\frac1{2\sqrt n}   \sum_{j=1}^{n-1}    \frac{j^{-1/2}}{-\Gamma_j(-1/2)}\biggr]\\
&\sim \sqrt{2n} \cdot  \sqrt 2\bigl( \tfrac14 \alpha^{-1}-1\bigr) \cdot  \frac{1}{-\Gamma(-1/2)}  \\
&=\sqrt{2n} \cdot  \sqrt 2\bigl( \tfrac14 \alpha^{-1}-1\bigr) \cdot  \frac{1}{2\sqrt \pi}.
\end{align*}
The theorem is proved. 
\end{proof}

We close this section with a remark regarding the expected limit of the quenched mean process in the stationary case.

In the setting of Theorem \ref{th:fBM} the random average process $H$ from the proof of Theorem \ref{th:SHE} has the initial profile
	\[H_0(x)=\begin{cases}0&\text{if }x=0,\\ \displaystyle\sum_{y=1}^x f(T_{y,0}\w)&\text{if }x>0,\quad\text{and}\\[12pt]\displaystyle\sum_{y=x+1}^0 E^\w f(T_{y,0}\w)&\text{if }x<0.\end{cases}\]
Thus, to extend the convergence result of Theorem \ref{th:SHE} to include the stationary case we need to prove a functional central limit theorem for the partial sums $\sum_{y=1}^x f(T_{y,0}\w)$.
 
 Finally, note that if indeed the claim of Theorem \ref{th:SHE} holds in the stationary setting of Theorem \ref{th:fBM}, then we would have
 	\[\sigma_0^2=\Vvv(f)=\frac1\beta-1=\frac1{4\alpha}-1=\frac{1/4-\alpha}{\alpha}=\frac{\rho_0^2\sigma_D^2}{\alpha}.\]
(Recall that we assumed the mean $\rho_0=1$ in Theorem \ref{th:fBM}.)
Thus one would have 
	\begin{align*}
	\mE[z(s,0)z(t,0)]
	&=\frac{\rho_0^2\sigma^2_D}{\alpha}\frac{\sigma}{\sqrt{2\pi}}(\sqrt t+\sqrt s-\sqrt{\abs{t-s}}\,)\\
	&=\frac{1}{2\sqrt{2\pi}}\bigl( \tfrac14 \alpha^{-1}-1\bigr)(\sqrt t+\sqrt s-\sqrt{\abs{t-s}}\,),
	\end{align*}
as stated in Theorem \ref{th:fBM}.

\bibliography{CurrentPaper}
\bibliographystyle{plain}
\end{document}